\documentclass[11pt,oneside]{amsart}
\usepackage{amsmath, amssymb,amsfonts,amscd}

\usepackage{xcolor}

\newtheorem{theorem}{Theorem}[section]
\newtheorem{corollary}[theorem]{Corollary}
\newtheorem{exercise}[theorem]{Exercise}

\newtheorem{lemma}[theorem]{Lemma}
\newtheorem{prop}[theorem]{Proposition}

\theoremstyle{definition}
\newtheorem{definition}[theorem]{Definition}

\newcommand{\PSH}{{\rm PSH}}
\newcommand{\idd}{{\bf 1}}

\newcommand{\capa}{{\rm Cap}}

\newcommand{\Tr}{{\rm Tr}}

\newcommand{\bR}{\mathbb{R}}
\newcommand{\R}{\mathbb{R}}
\newcommand{\Q}{\mathbb{Q}}

\newcommand{\f}{\varphi}
\newcommand{\p}{\psi}
\newcommand{\eps}{\varepsilon}
\newcommand{\e}{\varepsilon}

\numberwithin{equation}{section}

 \usepackage{hyperref}
\hypersetup{
    unicode=false,        
    pdftoolbar=true,      
    pdfmenubar=true,       
    pdffitwindow=false,     
    pdfstartview={FitH},    
    pdftitle={Singular CY metrics},    
    pdfauthor={E. DiNezza and C.H. Lu},     
    colorlinks=true,       
   linkcolor=black,          
    citecolor=black,        
    filecolor=black,      
    urlcolor=black}

\begin{document}
 
\title[Singular Calabi-Yau metrics]{Singular Calabi-Yau metrics} 
\author{Eleonora Di Nezza}

\address{IMJ-PRG, Sorbonne Universit\'e \& DMA, \'Ecole Normale Sup\'erieure, Universit\'e PSL,  CNRS, Paris, France.}
\email{\href{mailto:eleonora.dinezza@imj-prg.fr}{eleonora.dinezza@imj-prg.fr}, \href{mailto:edinezza@dma.ens.fr}{edinezza@dma.ens.fr}}
\urladdr{\href{https://perso.pages.math.cnrs.fr/users/eleonora.di-nezza/}{https://perso.pages.math.cnrs.fr/users/eleonora.di-nezza/}}

\author{Chinh H. Lu}

\address{Institut Universitaire de France \& Univ Angers, CNRS, LAREMA, SFR MATHSTIC, 2 Bd de Lavoisier, 49000 Angers,  France.}

\email{\href{mailto:hoangchinh.lu@univ-angers.fr}{hoangchinh.lu@univ-angers.fr}}
\urladdr{\href{http://hoangchinh-lu.perso.math.cnrs.fr}{http://hoangchinh-lu.perso.math.cnrs.fr}}
%\date{\today}

\maketitle

\begin{abstract}
	These are notes of lectures given by the first named author during the CIME Summer school Calabi-Yau varieties. We survey known results concerning the complex Monge-Amp\`ere equations in Hermitian contexts obtained by many authors during the last fifteen years.
\end{abstract}

\tableofcontents

\section{Introduction}

The Monge-Amp\`ere operator appeared in K\"ahler geometry since the very beginning of the study of K\"ahler manifolds.
The Ricci curvature tensor of a K\" {a}hler manifold can indeed be expressed by using a complex Monge-Amp\`ere operator, which explains why the existence problem for K\" {a}hler-Einstein metrics reduces to the study of a complex Monge-Amp\`ere equation.
An impressive number of works have been devoted to these questions in the past fifty years, following Yau's solution of the Calabi conjecture in the late 70's \cite{Yau78}.

Pluripotential theory in the setting of compact K\"ahler manifolds has proven to be a very effective tool in the study of degeneration of those canonical metrics in geometrically motivated problems (see for example \cite{Kol98} \cite{BEGZ10} \cite{EGZ09}). Usually in such degenerate settings singular K\"ahler metrics do appear as limits of smooth ones. Then pluripotential theory provides a natural background for defining the volume forms associated to such metrics. More importantly it provides useful information on the behavior of the K\"ahler potentials. On the other hand the theory does not rely on strong geometric assumptions, as most of the results are local in nature. It is therefore natural to expect that at least some of the methods and applications carry through in the more general Hermitian setting.

The interest towards Hermitian versions of the complex Monge-Amp\`ere equation has grown rapidly in the last fifteen years. The first steps were laid down by the French school most notably by Cherrier \cite{Cher87}. In this paper the author followed Yau’s arguments \cite{Yau78} to get existence of smooth solutions of the Monge-Amp\`ere equation in the case of smooth data  under some geometric assumptions on the background hermitian metric. The main problem to overcome was to establish a priori estimates needed for the closedness part in the continuity method. The renewed interest towards Hermitian Monge-Amp\`ere equations came with the paper by Guan-Li \cite{GL10} and especially the ones by Tosatti-Weinkove \cite{TW10}, \cite{TW10a}. Guan and Li were able to solve the equation assuming geometric conditions (different than those from \cite{Cher87}), while the missing  $C^0$-uniform estimate was finally established without any assumptions in \cite{TW10}.

The main challenge in working with non-K\"ahler forms is the appearance of torsion terms from the application of Stokes' formula, which are difficult to control. Despite this, Dinew and Ko{\l}odziej \cite{DK12} have succeeded in establishing a weak form of the comparison principle referred to as the ``modified comparison principle''. Building on this, they derived a priori $L^\infty$-estimates for densities in $L^p$ with $p>1$, generalizing the estimates of Tosatti and Weinkove. The existence of weak bounded solutions was subsequently proven in the works of Ko{\l}odziej and Nguyen \cite{KN15Phong}, \cite{Ng16AIM}, \cite{KN19}. 

Over the past fifteen years, various methods for obtaining $L^{\infty}$-estimates have been developed by many authors.  B{\l}ocki \cite{Bl05China}, \cite{Bl11China} employed the Alexandroﬀ-
Bakelman-Pucci maximum principle and a local stability estimate due to Cheng and Yau. Sz\'ekelyhidi \cite{Sze18} extended B{\l}ocki's method, studying non-linear elliptic PDEs on compact Hermitian manifolds. In \cite{GPT23}, a PDE approach was introduced in the K\"ahler context, using an auxiliary Monge-Amp\`ere equation, drawing inspiration from the work of Chen and Cheng on constant scalar curvature metrics \cite{ChCh21a}. The same method was shown to be effective in the non-K\"ahler setting and for various other Hessian equations in \cite{GP24}.  Furthermore, in  \cite{GPTW}, it is shown that the method of \cite{GPT23} also applies in several degenerate settings including nef classes.

Significant contributions have also been given very recently in a series of papers \cite{GL24}, \cite{GL22}, \cite{GL23Crelle}, \cite{BGL24}. In particular, in \cite{GL23Crelle} and \cite{BGL24} the authors develop a new approach to $L^\infty$-estimates for degenerate complex Monge-Amp\`ere equations on complex manifolds, which has a ``local" nature and mainly relies on envelopes.

In this note we will follow the approach developed in \cite{GL23Crelle} and \cite{BGL24}. All results presented here are known, and we do not claim any originality. 

\noindent {\bf Acknowledgment.} We thank Q.T. Dang and the anonymous referees for carefully reading the note and giving many useful comments.

\section{The complex Monge-Amp\`ere operator}
We recall basic tools in pluripotential theory and direct the interested reader to the textbook \cite{GZbook} for a complete treatment. 
In the whole section, $\Omega$ is a bounded domain in $\mathbb C^n$.  

\subsection{Positive forms and positive currents}

Consider a real oriented manifold $M$ of dimension $m$.
Recall that a \emph{current} $T$ of degree $q$ (or dimension $m-q$) on $M$ is a continuous linear form on the vector space $\mathcal{D}^{m-q}(X)$ of smooth differential forms of
degree $m-q$ with compact support. We denote by $\mathcal{D}'_{m-q}(M)$  (or $ {\mathcal{D}^{'}}^{q}(M)$) the space of currents of degree $m-q$ on $M$, and by $\langle T, \eta \rangle$ the pairing between a test $m-q$-form $\eta$ and a current $T$ of degree $q$.
A first example of a current of degree $q$ is the \emph{current of integration} over a smooth closed oriented submanifold $Z$ of dimension $m-q$, which is denoted by $[Z]$ and defined as
 $$\langle [Z], \eta \rangle:=\int_Z \eta.$$
Observe that, given $\tau$ a $(m-q)$-form with coefficients in $L^1_{loc}(M)$, we can associate the current $T_\tau$ of degree $q$ defined as follows:
 $$\langle T_\tau, \eta \rangle:=\int_M \tau\wedge \eta.$$
As a fact, let us mention that, if we let $(x_1, \cdots, x_m)$ be local coordinates on an open subset $\Omega\subset M$, then every current of degree $q$ can be written in an unique way as 
\begin{equation}\label{eq: current}
T=\sum_{|I|=q} T_I dx_I,
\end{equation}
where $T_I$ are distributions of order $s$ on $\Omega$.
 
Given a current $T$ of degree $q$, the wedge product of $T$ with a smooth $p$-form $\beta$ is defined as
$$\langle T \wedge \beta, \eta \rangle:=\langle T, \beta \wedge \eta \rangle .$$
One can also define the exterior derivative $dT$ as the $(q+1)$-current satisfying
$$\langle dT, \eta \rangle:=(-1)^{q+1}\langle T, d\eta \rangle.$$ 
A current $T$ is then said to be \emph{closed} if $dT = 0$. We denote by $\{T\}$ the cohomology class
defined by the current $T$. By deRham's Theorem, the corresponding
cohomology vector space
$$H^{m-q}(M) :=\{{\rm{closed\, currents\, of \,degree}}\, q \} / \{ dS \, | \, S\, \mbox{current of degree}\, q-1 \}$$
is isomorphic to the one defined using closed smooth differential forms.

Let now $X$ be a complex manifold of complex dimension $n$. The decomposition of complex valued differential forms according to their bidegrees induces a decomposition at the level of currents. We say that a current $T$ is of bidegree $(p,q)$ if it is of degree $p+q$ and $\langle T, \eta\rangle=0$ for any test form $\eta$ of bidegree $(k,l)\neq (n-p, n-q)$. We denote by ${\mathcal{D}'}^{p,q}(X)$ the space of such currents, and by $H^{p,q}(X)$ the corresponding vector space of cohomology classes.\\ \indent In the complex case one can define a notion of positivity at the level of forms and currents.
Let $V$ be a complex vector space of dimension $n$ and $(z_1,...,z_n)$ coordinates on $V$. Observe that $V$ has a canonical orientation defined by the volume form
$$(idz_1\wedge d\bar{z}_1)\wedge...\wedge (idz_n\wedge d\bar{z}_n)$$ 
and a $(n,n)$-form on $V$ is said to be positive if and only if it is a positive multiple of the orientation form. A $(p,p)$-form $\eta$ is said to be \emph{positive} if for all $\alpha_j\in V^*$, $1\leq j\leq n-p$, we have that
$$\eta \wedge (i \alpha_1\wedge \bar{\alpha_1})\wedge...\wedge (i \alpha_{n-p}\wedge \bar{\alpha}_{n-p})$$
is a positive $(n,n)$-form. Equivalently, a form of bidegree $(p,p)$ is positive if and only if its restriction to every $p$-dimensional subspace $S\subset V$ is a positive volume form on $S$.\\
\indent The set of positive $(p,p)$-forms is a closed convex cone in $\bigwedge^{p,p}V^*$ and its dual cone in $\bigwedge^{n-p,n-p}V^*$ is the \emph{strongly positive cone}. A strongly positive $(q,q)$-form $\beta$ is a convex combination of forms of type
$$(i \alpha_1\wedge \bar{\alpha_1})\wedge\cdots \wedge (i \alpha_q\wedge \bar{\alpha_q})$$ 
with $\alpha_j\in V^*$, for $j=1,\cdots,q$. Of course, we are interested in the case $V=T_x X$. In this way, we are able to define, at each $x\in X$, a notion of positivity for smooth forms on $X$, and so we can then give a notion of positivity for currents. A current $T$ of bidimension $(p, p)$ is \emph{positive} if $\langle T, \beta\rangle\geq 0$ for all strongly positive test forms $\beta\in \mathcal{D}^{p,p}(X)$.\\
\indent Two extreme examples of positive currents are currents of integration along analytic subsets of dimension $p$ and  positive smooth differential forms of bidegree $(n-p, n-p)$.\\
Again as a fact, we mention that, any positive $(p,p)$-current $T$ is a real current of order $0$, i.e. in \eqref{eq: current} the coefficients are distributions of order $0$. More precisely, $T$ can be written (in local holomorphic coordinates in $\Omega\subset X$) as 
\begin{equation}\label{eq current1}
T=i^{(n-p)^2} \sum_{|I|=|J|=p} T_{I,J} dz_I \wedge d\bar{z}_J,
\end{equation}
where the coefficients $T_{I,J}$ are complex measures in $\Omega$ satisfying $\bar{T}_{I,J}=T_{J,I}$ and $T_{I,I}$ are positive Borel measures in $\Omega$.

\medskip
\indent Now, let $T$ be a positive closed current of bidegree $(1,1)$ on a compact manifold $X$. Then $T$ is locally given as $T=dd^c \f$ where $\f$ is a plurisubharmonic (psh for short) function. Note that, this cannot hold globally since the maximum principle ensures that the only psh functions on $X$ are the constants. 

In the text we use the normalization $d^c =  i(\bar \partial -\partial)$, so that $dd^c =2i\partial \bar \partial$.

%On the other hand, given $\theta$ a smooth representative of $\{T\}$, one can ask whether $T-\theta$ (which is $d$-exact) is also $dd^c$-exact. This is true in the K\"ahler setting:
%\begin{lemma}[$\partial\bar{\partial}$-Lemma]
%Let $X$ be a compact K\"ahler manifold. Let $S$ be a current which is both $\partial$ and $\bar{\partial}$-closed. Then $S$ is $d$-exact if and only if it is $dd^c$-exact.
%\end{lemma}
%\textcolor{red}{We never use this lemma in the note since our manifold is not Kahler. }
%\noindent We refer the reader to \cite[Proposition 6.17]{cv} for a proof of the $\partial\bar{\partial}$-Lemma and to \cite{Dem} for more details about the notion of positive currents.\\

%Consider now $\alpha\in H^{1,1}(X, \mathbb{R})$ a cohomology class which can be represented by a positive closed current (such a class is 
%called \emph{pseudoeffective}). Fix $\theta$ a smooth representative of $\alpha$. If $X$ is K\"ahler, then any closed positive current of bidegree $(1,1)$ in $\alpha$ can be written as
%$$T=\theta+dd^c\varphi$$ 
%for some upper semi-continuous function $\f:X\rightarrow\mathbb{R} \cup \{-\infty\}$, which is uniquely determined up to an additive constant. Such functions are called $\theta$-plurisubharmonic.
\subsection{Plurisubharmonic functions} 
A function $u : \Omega \rightarrow \mathbb{R} \cup \{-\infty\}$ is called  plurisubharmonic (or psh) if it is upper semicontinuous and its restriction on each complex line intersecting $\Omega$ is subharmonic. In other words, the latter means that for all $z_0\in \Omega$, $w\in \mathbb{C}^n$, $r>0$ such that $z_0+tw \in \Omega$, for all $|t|\leq r$, we have 
\[
u(z_0) \leq \frac{1}{2\pi} \int_0^{2\pi} u(z_0+ rwe^{i\theta})d\theta. 
\]
We denote by $\PSH(\Omega)$ the set of plurisubharmonic functions. Such functions enjoy the following basic properties. 
\begin{prop}
\begin{enumerate}
  \item  If $u$ and $v$ are psh and $u=v$ almost everywhere then $u=v$ everywhere. 
		\item If $u$ and $v$ are psh then $u+v$, $\lambda u$ (with $\lambda>0$), $\max(u,v)$ are psh. 
  \item If $u$ is psh on $\Omega$ and if $\chi : I \rightarrow \mathbb{R}$ is a convex increasing function where $I$ is an interval containing $u(\Omega)$, then $\chi \circ u$ is psh as well.
    
		\item If $u$ is psh then the regularization $u\star \rho_{\eta}$ by convolution with smoothing kernels provides smooth psh functions decreasing to $u$. 
\item If $(u_k)$ is a decreasing sequence of psh functions, then their limit $u = \inf_k u_k$ is psh as well.

\item Consider $u$ psh on $\Omega$ and $v$ psh on a relatively compact subdomain $\Omega'  \subset \Omega$. If $u \geq v$ on $\partial \Omega'$ then the function $\max(u, v)$ (defined as simply $u$ outside of $\Omega'$) is psh on $\Omega$. 
%\item If $u$ is psh and not identically $-\infty$ in $\Omega$, then $e^{-\varepsilon u}$ is integrable for some $\varepsilon>0$ small enough. In particular, $u\in L^p(\Omega)$ for all $p>0$.
\end{enumerate}
\end{prop}

The property in (6) will be useful whenever we want to prove a local property. It is indeed convenient to modify the involved  functions near the boundary in order not to have boundary terms when we do apply integration by parts. For this, assume $u$ is psh and bounded $-M\leq u\leq -1$ in $\mathbb B=B(0,1)$. Let $A>1$. Then the function $\max(u,A(|z|^2-1))$ is psh in $\mathbb B$, it coincides with $u$ in $B(0,1/2)$ (if $3A \geq 4M$) and with $A(|z|^2-1)$ in $\mathbb B \setminus B(0,(1-1/A)^{1/2})$, since on $\partial B(0,(1-1/A)^{1/2})$ we have $A(|z|^2-1)= -1 \geq u$. Thus the new psh function $\max(u,A(|z|^2-1))$ is zero on $\partial \mathbb{B}$.
\medskip

%A bounded domain $\Omega\subset \mathbb C^n$ is strongly pseudoconvex  if there exists a smooth strictly psh function $\rho$ defined in a neigborhood of $\bar \Omega$ such that $\partial \Omega = \{\rho=0\}$ and $\nabla \rho \neq 0$ on $\partial \Omega$. In particular, the sublevel sets $\{x\in \Omega\; : \; \rho(x)<-c\}$ is relatively compact in $\Omega$ for all $c>0$. 

For $u : \mathbb{C} \rightarrow \mathbb{R} \cup \{- \infty\}$ a $C^2$-smooth function, subharmonicity is equivalent to $\dfrac{\partial^2u}{\partial z \partial \bar{z}} \geq 0$. In higher dimension and for non-smooth functions, the following result holds: %GZ2016, 1.43
\begin{prop} \label{gz1.43}
   If $u \in \PSH(\Omega)$, then for any $\xi \in \mathbb{C}^n$, 
   \[\sum_{1\leq j, k\leq n} \xi_j \bar{\xi_k} \dfrac{\partial^2u}{\partial z_j \partial \bar{z}_k} \geq 0\]
   is a positive distribution in $\Omega$. 
   
    Conversely, if $U \in \mathcal{D}'(\Omega)$ is a distribution such that for all $\xi \in \mathbb{C}^n$, the distribution $\sum_{j, k}\xi_j \bar{\xi_k} \dfrac{\partial^2U}{\partial z_j \bar{z_k}}$ is positive, then there exists a unique $u \in \PSH(\Omega)$ such that $U \equiv T_u$. 
\end{prop}

\begin{proof}
    Fix $\xi \in \mathbb{C}^n$ and consider the operator 
$$\Delta_\xi : u \mapsto \sum_{j, k} \xi_j \bar{\xi_k} \dfrac{\partial^2u}{\partial z_j \partial \bar{z}_k}$$
defined on all $\mathcal{C}^2$ functions on $\Omega$. Fix $a \in \Omega$ and $u$ a smooth plurisubharmonic function on $\Omega$. 
    Consider the function $u_\xi : z \mapsto u(a + z\xi)$ defined over a neighborhood of $0$ in the complex line. Where it is defined, $u_\xi$ is $\mathcal{C}^2$ and 
$$\dfrac{\partial^2u_\xi}{\partial z \partial \bar{z}}(z_0) = \Delta_\xi u(a + z_0\xi).$$
    We have thus brought the problem down to subharmonic functions on a domain in $\mathbb{C}$. It is clear from the integral definition of subharmonicity that $\dfrac{\partial^2u_\xi}{\partial z \partial \bar{z}} \geq 0$. 
    
    \medskip

    When $u$ is not smooth, we can regularize it by convolution. Because $\Delta_\xi$ is linear with constant coefficient, it commutes with convolution and so  $\sum_{j, k} \xi_j \bar{\xi_k} \dfrac{\partial^2u}{\partial z_j \partial \bar{z_k}}$ is a positive distribution. For more details on plurisubharmonic smoothing see \cite[Section 1.3.3.1, pp. 24-26]{GZbook} 

   Consider $\rho$ a smooth non-negative radial function on $\mathbb{C}^n$ with support in the unit ball and whose integral on $\mathbb{C}^n$ is $1$. We then define 
$$\rho_{x, \eps} : z \mapsto \dfrac{1}{\eps^{2n}}\rho\left(\dfrac{z - x}{\eps}\right).$$
   For $U$ a distribution on $\Omega$, we want to define $u$ by $u(x) = \lim_{\eps \rightarrow 0} U(\rho_{x, \eps})$. Assume the distribution $\sum_{j, k}\xi_j \bar{\xi_k} \dfrac{\partial^2U}{\partial z_j \bar{z_k}}$ is positive for every $\xi$, and for every $\eps > 0$ consider $u_\eps$ the function defined within $\Omega$ on points that are at a distance of at least $\eps$ of the boundary of $\Omega$, such that $u_\eps(x) = U(\rho_{x, \eps})$. The function $u_\eps$ is smooth and plurisubharmonic, and by $u_\eps$ decreases to a limit $u$ as $\eps$ decreases to $0$. Therefore $u$ satisfies the mean value inequalities and its associated distribution is $U$.
\end{proof}

We can then reformulate the description of psh functions:
\begin{prop}
    If $u\in \PSH(\Omega)$, then $dd^cu$ is a positive $(1, 1)$-current.
\end{prop}

\subsection{Weak convergence of measures and currents}

Recall that a sequence of Radon measures $\mu_j$ converges weakly in $\Omega$ to a Radon measure $\mu$ if for all continuous function $\chi$ with compact support in $\Omega$ we have 
\[
\lim_{j\to +\infty}\int_{\Omega} \chi d\mu_j  = \int_{\Omega} \chi d\mu. 
\]
We will use several times the following elementary result. 
\begin{lemma}\label{lem: semicontinuity}
	Assume $(\mu_j)$ is a sequence of positive Radon measures converging weakly to a positive Radon measure $\mu$ in $\Omega$. If $(f_j)$ is a sequence of positive lower semicontinuous functions in $\Omega$ which increase to $f$, then 
	\[
	\liminf_{j\to +\infty} \int_{\Omega} f_j d\mu_j \geq \int_{\Omega} fd\mu. 
	\] 
	In particular, if $U\subset \Omega$ is open then 
	\[
	\liminf_{j\to +\infty} \int_{U} d\mu_j \geq \int_{U} d\mu. 
	\]
\end{lemma}
The last statement follows from the previous one because the function ${\bf 1}_U$ is lower semicontinuous. 
\begin{proof}
	Assume first that $f_j=f$, and let $(g_j)$ be a sequence of positive continuous functions increasing to $f$. Then, for all $j,k$, 
	\[
	\liminf_{j\to +\infty}\int_{\Omega} f d\mu_j  \geq  \liminf_{j\to +\infty}\int_{\Omega} g_k d\mu_j =\int_{\Omega} g_kd\mu.
	\]
	Letting $k\to +\infty$, we obtain the inequality in this case. For the general case, since $f_j\geq f_k$ when $j\geq k$, we have, by the previous step, 
	\[
	\liminf_{j\to +\infty}\int_{\Omega} f_j d\mu_j \geq \liminf_{j\to +\infty}\int_{\Omega} f_k d\mu_j \geq \int_{\Omega}f_k d\mu.
	\]
	The conclusion follows by letting $k\to +\infty$ and using monotone convergence. 
\end{proof}

\begin{lemma}\label{lem: semi cont compact}
	Assume $(\mu_j)$ is a sequence of positive Radon measures converging weakly to a positive Radon measure $\mu$ in $\Omega$. If $K\Subset \Omega$ is compact then
	\[
	\limsup_{j\to +\infty} \int_K d\mu_j \leq \int_K d\mu. 
	\] 
\end{lemma}
\begin{proof}
	Fix a smooth positive function $\chi$ with compact support in $\Omega$ such that $\chi=1$ on $K$. Then 
	\[
	\limsup_{j\to +\infty} \int_K d\mu_j \leq \limsup_{j\to +\infty}\int_{\Omega} \chi d\mu_j = \int_{\Omega} \chi d\mu. 
	\] 
	Letting $\chi$ decrease to $\idd_K$, we obtain the result. 
\end{proof}

\subsection{Monge-Amp\`ere operator}
Let $\Omega$ be an open set in $\mathbb C^n$. If $T$ is a closed positive $(p,p)$-current then it can be written as a differential form with measures coefficients, as in \eqref{eq current1}. If $u$ is bounded and psh then $uT$ is a current and we can take $dd^c(uT)$.

If $u$ is bounded, then one can define $uT$ which is the current with coefficients $uT_{I,J}$. When $u$ is smooth, $\langle uT, \eta\rangle:= \langle T, u\eta \rangle $ and $dd^c u \wedge T$ is a $(p+1, p+1)$-current in $\Omega$ defined as $\langle dd^c u \wedge T, \eta \rangle=\langle T, dd^c u \wedge \eta \rangle $. Observe that if $u$ is smooth and psh, by definition $dd^c u \wedge T$ is still a positive and closed current.

Now observe that if we set $\Theta:= d^c u \wedge \eta -u d^c \eta$, then 
\begin{eqnarray*}
d\Theta&=&dd^c u \wedge \eta -d^c u \wedge d\eta -du \wedge d^c \eta- udd^c \eta\\
&=& dd^c u \wedge \eta +2 \partial u \wedge \partial\eta -2 \bar{\partial}u \wedge \bar{\partial} \eta- udd^c \eta.
\end{eqnarray*}
Thus
\begin{eqnarray*}
\langle dd^c u \wedge T, \eta \rangle&=& \langle T, dd^c u \wedge \eta \rangle \\
&=& \langle T,  u  dd^c\eta \rangle + \langle T,  d\Theta \rangle + 2 \left(  \langle T, \bar{\partial}u \wedge \bar{\partial} \eta  \rangle -  \langle T, {\partial}u \wedge {\partial} \eta  \rangle  \right)\\
&=&  \langle uT,  dd^c\eta \rangle =  \langle dd^c(uT),  \eta \rangle,
\end{eqnarray*}
where the last line follows from the fact that $T$ is closed and that $\langle T, \bar{\partial}u \wedge \bar{\partial} \eta  \rangle =  \langle T, {\partial}u \wedge {\partial} \eta  \rangle=0 $ for degree reasons.

This motivates the following definition:
given $T$ a closed and positive $(p,p)$-current and $u\in \PSH(\Omega) \cap L^\infty (\Omega)$, then the current $dd^c u\wedge T$ is the $(p+1, p+1)$-current defined as 
$$dd^c u \wedge T= dd^c(uT).$$

%\subsection{Convergence of the Monge-Amp\`ere operator}
\begin{prop}[Integration by parts]\label{prop: integration by parts}
	Let $T$ be a closed positive current of bidegree $(n-1,n-1)$. Assume $u, v\in \PSH(\Omega)$ are bounded, $u=0$ on $\partial \Omega$, $v\leq 0$, and $\int_{\Omega} dd^c v \wedge T <+\infty$. Then 
	\[
	\int_{\Omega} u dd^c v \wedge T \geq \int_{\Omega} v dd^c u \wedge T. 
	\]
	In particular, if $v=0$ on $\partial \Omega$ and $\int_{\Omega} dd^c u\wedge T<+\infty$, then the above inequality is an equality. 
\end{prop}
\begin{proof}
	Fix $\varepsilon>0$ and set $u_{\varepsilon}=\max(u,-\varepsilon)$. Then $u_{\varepsilon}-u$  vanishes outside $D$, where $D$ is a relatively compact subset of $\Omega$. The convolution $(u_{\varepsilon}-u)*\rho_{\eta}$ is compactly supported in $D$ for $\eta>0$ small enough. Let $(v_k)$ be a sequence of smooth psh functions defined in an open neighborhood of $\bar D$ which decreases to $v$. By definition of $dd^c v_k \wedge T$ and integration by parts we have 
	\begin{flalign*}
		\int_D (u_{\varepsilon}-u)*\rho_{\eta} \, dd^c v_k\wedge T &= \int_D (-v_k) dd^c \left((u-u_{\varepsilon})*\rho_{\eta} \right)\wedge T\\
		& \leq \int_D (-v_k) dd^c u*\rho_{\eta} \wedge T,
	\end{flalign*}
 where the last line follows from the fact that $dd^c (u_\varepsilon * \rho_\eta)$ is a positive current.
	Letting $\eta\to 0^+$, and noting that ${\bf 1}_{\bar D} (-v_k)$ is upper semicontinuous in $\Omega$, we obtain 
	\[
	\int_D (u_{\varepsilon}-u) dd^c v_k \wedge T \leq \int_{\bar D} (-v_k) dd^c u \wedge T, 
	\]
	which, since $-u_\varepsilon \leq \varepsilon$, gives 
	\[
	\int_D (-u) dd^c v_k \wedge T \leq  \int_{\bar D} (-v_k) dd^c u \wedge T + \varepsilon \int_{\bar D} dd^c v_k \wedge T. 
	\]
	Letting $k\to +\infty$, then $D\to \Omega$ and finally $\varepsilon\to 0^+$ we obtain the desired result.  
	\end{proof}
	
	\begin{prop}\label{prop: symm}(Symmetry of the Monge-Ampère operator)
		Assume $T$ is a closed positive current of bidegree $(n-2,n-2)$ and $u, v$ are bounded psh functions. 
  Then 
		\[
		dd^c u \wedge dd^c v \wedge T = dd^c v \wedge dd^c u \wedge T. 
		\]
	\end{prop}
	\begin{proof}
%	If both $u$ and $v$ are smooth then the equality is trivial. Assume that $u$ is smooth and let $v_j$ be a decreasing sequence of smooth psh functions converging to $v$. 
 %Then 
	%\[
	%dd^c u \wedge dd^c v_j \wedge T
	%\]
		Since the problem is local, we can assume w.l.o.g. that $\Omega$ is the unit ball and that $u=v=0$ on the boundary of $\Omega$. Let $\psi$ be a smooth psh function in $\Omega$ that vanishes on $\partial \Omega$. We claim that 
		\[
		\int_{\Omega} \psi dd^c u \wedge dd^c v \wedge T =  \int_{\Omega} \psi dd^c v \wedge
		dd^c u  \wedge T. 
		\]	
	Indeed, integrating by parts using Proposition \ref{prop: integration by parts} gives
		\[
		\int_{\Omega} \psi dd^c u \wedge dd^c v \wedge T =  \int_{\Omega} u dd^c \psi \wedge
		dd^c v  \wedge T. 
		\]	
		Since $\psi$ is smooth, we also have $dd^c \psi  \wedge
		dd^c v \wedge T=
		dd^c v \wedge dd^c \psi  \wedge T$, yielding 
		\[
		 \int_{\Omega} u dd^c \psi \wedge
		dd^c v \wedge T = \int_{\Omega} u 
		dd^c v \wedge dd^c \psi \wedge T.
		\]	
		Integrating by parts again (Proposition \ref{prop: integration by parts}) and repeating the previous argument (where the role of $u$ and $v$ are reversed) we obtain 
		\[
		\int_{\Omega} u 
		dd^c v \wedge dd^c \psi \wedge T =\int_{\Omega} v
		dd^c u  \wedge dd^c \psi \wedge T=\int_{\Omega} \psi dd^c v \wedge dd^c u \wedge T, 
		\]
		proving the claim. 
		
		Now, given $\chi$ a smooth test function (with compact support) in $\Omega$, we decompose $\chi=\psi_1-\psi_2$, where $\psi_1=A\rho + \chi$ and $\psi_2=A\rho$, where $\rho$ is the defining function of $\Omega$. Thus $\psi_1, \psi_2$ are both smooth, psh in $\Omega$, and vanish on $\partial\Omega$. By the claim we have  
		\[
		\int_{\Omega} \chi dd^c u \wedge dd^c v \wedge T =  \int_{\Omega} \chi dd^c v \wedge
		dd^c u  \wedge T,
		\]	
		finishing the proof. 
		\end{proof}

Using integration by parts and symmetry of the Monge-Amp\`ere operator, we see that $dd^c (uT)=dd^c u\wedge T$ is still a positive current:

\begin{prop}\label{prop:cont MA}
 Let $u_j$ be smooth psh functions in $\Omega$ decreasing to $u\in \PSH(\Omega) \cap L^\infty (\Omega)$. Then
 $$dd^c u_j \wedge T \quad {\rm weakly \,converges\, to} \quad dd^c u \wedge T.$$
 In particular $dd^c u\wedge T $ is a positive and closed current.
\end{prop}
\begin{proof}
By definition $dd^c u \wedge T=dd^c(uT)$ hence it is closed. Moreover, since $u_j \searrow u$ (hence $u_j$ converges to $u$ in $L^1(\Omega)$), we know that $u_j T $ converges in the weak sense of currents to $u T $. Hence $dd^c (u_j T)$ converges to $dd^c (uT)$ in the weak sense of currents. The positivity property is then preserved when passing to the limit. 
\end{proof}

The Monge-Amp\`ere operator can thus be defined inductively. More precisely, for $u_i\in \PSH(\Omega) \cap L^\infty (\Omega)$ we have
\begin{equation}\label{MA def}
dd^c u_1 \wedge ... \wedge dd^c u_p = dd^c (u_1 dd^c (u_2\cdots (u_{p-1}dd^c u_p))).  
\end{equation}
%In particular, $(dd^c u)^n$ is a positive Radon measure. 

The $(p,p)$-current above is well defined thanks to Proposition \ref{prop: symm} below, is closed (by definition) and positive thanks to the following result:

	\begin{theorem}[Continuity along decreasing sequences]\label{thm: BT convergence}
			Assume $(u_k^j)$,  $j\in \mathbb N$, $k=0,...,p\leq n$, are decreasing sequences of uniformly bounded psh functions which converge to $u_k$. Let $T$ be a closed positive $(q,q)$-current in $\Omega$, with $p+q\leq n$. Then the sequence $u_0^j dd^c u_1^j \wedge ... \wedge dd^c u_p^j \wedge T$ weakly converges to $u_0 dd^c u_1\wedge ... \wedge dd^c u_p \wedge T$. 
		\end{theorem}

% We refer to \cite[Theorem 3.18]{GZbook} for a proof. 
 
\begin{proof}
We proceed by induction on $p$. For $p=0$  we know already that $u_0^jT$ weakly converges to $u_0 T$. Fix $1 \leq p \leq n -q$ and assume that
the theorem is true for $p-1$, i.e. $$u_1^j dd^c u_2^j \wedge ... \wedge dd^c u_p^j \wedge T \rightharpoonup u_1 dd^c u_2\wedge ... \wedge dd^c u_p \wedge T. $$ As consequence
$$ S^j:=dd^c u_1^j \wedge ... \wedge dd^c u_p^j \wedge T \rightharpoonup  S:=dd^c u_1\wedge ... \wedge dd^c u_p \wedge T. $$
It follows from the Chern–Levine–Nirenberg inequality (see for e.g. \cite[Theorem 3.9.]{GZbook}) that the sequence $u_0^j S^j$ is relatively compact for the weak topology of currents. It then suffices to show that if the sequence  $u_0^j S^j$ converges weakly to a
current $\Theta$, then $\Theta=u_0 S$.
By upper semi-continuity we know that for all positive forms $\Gamma$ (of the correct degree), $\Theta\wedge \Gamma \leq u_0S \wedge \Gamma$. In particular $u_0S-\Theta$ is a positive current. It thus remains to prove that its mass is zero, or equivalently that on every small ball $\mathbb{B}:=\mathbb{B}(a, R) \subset \Omega$ the following holds:
\begin{equation}\label{zero mass}
 \int_{\mathbb B} u_0S \wedge \beta^{n-(p+q)} \leq   \int_{\mathbb B} \Theta \wedge \beta^{n-(p+q)}, 
 \end{equation}
where $\beta=dd^c |z-a|^2.$
Up to rescaling the ball, we can also assume that $u_k^j$, and $u_k$, for all $j$ and all $k=1,\cdots, p$, coincide with the function $A(|z-a|^2-R^2)$ in a neighborhood of $\partial \mathbb{B}$, $A>>1$; in particular $u_k^j=0$ on $\partial \mathbb{B}$. Then $u_j^k$ are all negative in the ball by maximum principle. Integrating by parts (Proposition \ref{prop: integration by parts}) and using that $u_k^j \geq u_k$, we get
\begin{eqnarray*}
  \int_{\mathbb{B}} u_0 S \wedge \beta^{n-p-q} &=&  \int_{\mathbb{B}} u_0 \bigwedge_{1\leq k\leq p} dd^c u_k\wedge T\wedge \beta^{n-p-q}\\
&\leq & \int_{\mathbb{B}} u_0^j \bigwedge_{1\leq k\leq p} dd^c u_k\wedge T\wedge \beta^{n-p-q}\\
&= & \int_{\mathbb{B}} u_1\,dd^c u_0^j \bigwedge_{2\leq k\leq p} dd^c u_k\wedge T\wedge \beta^{n-p-q}\\
&\leq & \int_{\mathbb{B}} u_1^j\,dd^c u_0^j \bigwedge_{2\leq k\leq p} dd^c u_k\wedge T\wedge \beta^{n-p-q}\\
&\leq & \int_{\mathbb{B}}  u_0^j \bigwedge_{1\leq k\leq p} dd^c u^j_k\wedge T\wedge \beta^{n-p-q}\\ 
&\leq & \int_{\mathbb{B}}  u_0^j S^j\wedge T\wedge \beta^{n-p-q}.
\end{eqnarray*}
Observe that the third inequality follows by applying Proposition \ref{prop: symm} and Proposition \ref{prop: integration by parts} $(p-1)$ times.
Also, since the positive measures $(-u_0^j) S^j \wedge \beta^{n-p-q}$ converge weakly to $-\Theta\wedge  \beta^{n-p-q}$ we have that 
$$\liminf_{j\rightarrow +\infty} \int_{\mathbb{B}} (-u_0^j) S^j \wedge \beta^{n-p-q} \geq  \int_{\mathbb{B}} -\Theta  \wedge \beta^{n-p-q}.$$
Combining the above we get \eqref{zero mass}, that is what we wanted.
\end{proof}

\begin{theorem}[Plurifine locality]
	Assume $\alpha$ is a smooth $(q,q)$-form in $\Omega$, $T$ is a closed positive  $(p,p)$-current  and $u,v$ are a bounded psh functions in $\Omega$. Then, for $r= n-p-q$, 
	\[
	{\bf 1}_{\{u >v\}} (dd^c \max(u,v))^r \wedge T \wedge \alpha = {\bf 1}_{\{u >v\}} (dd^c u )^r \wedge T \wedge \alpha.
	\]
\end{theorem}
\begin{proof}
	If $u$ is continuous then the set $\{u>v\}$ is open in $\Omega$ hence the result is trivial. Let $u_j$ be a sequence of smooth psh functions decreasing to $u$ and set $f_j= \max(u_j,v)-v$, $f=\max(u,v)-v$. Since $f_j=0$ if $u_j \leq v$, by the above observation we have that
	\[
	f_j (dd^c \max(u_j,v))^r \wedge T \wedge \alpha =f_j (dd^c u_j)^r \wedge T \wedge \alpha.
	\]
	By Bedford-Taylor's convergence theorem (Theorem \ref{thm: BT convergence}), letting $j\to +\infty$ we obtain 
	\[
	f (dd^c \max(u,v))^r \wedge T \wedge \alpha =f(dd^c u)^r \wedge T \wedge \alpha,
	\]
 and in particular
 $$ 	{\bf 1}_{\{u >v\}} f (dd^c \max(u,v))^r \wedge T \wedge \alpha = 	{\bf 1}_{\{u >v\}}f(dd^c u)^r \wedge T \wedge \alpha.$$
	For each $\varepsilon>0$, we then have 
	\[
	\frac{f}{f+\varepsilon}	{\bf 1}_{\{u >v\}} (dd^c \max(u,v))^r \wedge T \wedge \alpha =\frac{f}{f+\varepsilon}	{\bf 1}_{\{u >v\}}(dd^c u)^r \wedge T \wedge \alpha.
	\]
	Letting $\varepsilon\to 0$, we arrive at the result. 
\end{proof}

\begin{corollary}[Maximum principle]\label{cor: max principle}
	Assume $u,v$ are bounded psh functions in $\Omega$. Then 
	\[
	(dd^c \max(u,v))^n \geq {\bf 1}_{\{u\geq v\}} (dd^c u)^n +  {\bf 1}_{\{u< v\}} (dd^c v)^n. 
	\]
 In particular, if $u\leq v$ then ${\bf 1}_{\{u= v\}} (dd^c u)^n \leq {\bf 1}_{\{u= v\}} (dd^c v)^n$. 
\end{corollary}
\begin{proof}
   For any $\varepsilon>0$ we have, using the plurifine locality above, 
    \begin{eqnarray*}
(dd^c \max(u+\varepsilon,v))^n &\geq &  {\bf 1}_{\{u+\varepsilon> v\}} (dd^c  \max(u+\varepsilon,v))^n  \\
&+&  {\bf 1}_{\{u+\varepsilon< v\}}  (dd^c\max(u+\varepsilon,v))^n\\
&=&{\bf 1}_{\{u+\varepsilon> v\}} (dd^c u)^n  +  {\bf 1}_{\{u+\varepsilon< v\}}  (dd^c v)^n \\
&\geq &{\bf 1}_{\{u\geq  v\}} (dd^c u)^n  +  {\bf 1}_{\{u+\varepsilon< v\}}  (dd^c v)^n.
	 \end{eqnarray*}
  Sending $\varepsilon$ to zero we get the first statement. \\
  In particular, if $u\leq v$ we have $\max(u,v)=v$, thus
  $$  {\bf 1}_{\{u\geq  v\}} (dd^c v)^n \geq {\bf 1}_{\{u\geq v\}} (dd^c u)^n. $$
  This is what we want since $\{u\geq v\}=\{u=v\}$.
\end{proof}

\begin{lemma}
If $u$ and $v$ are bounded psh functions in $\Omega$ and $u=v$ near the boundary $\partial \Omega$, then $\int_{\Omega} (dd^c u)^n = \int_{\Omega} (dd^c v)^n$. 
\end{lemma}
\begin{proof}
	If $u=v$ on $\Omega\setminus K$, where $K$ is a compact subset of $\Omega$, and $\chi$ is a smooth positive function which is identically $1$ in a neighborhood of $K$, then 
	\[
	\int_{\Omega} \chi[(dd^c v)^n- (dd^c u)^n]=\int_{\Omega} \chi (dd^c v \wedge T -dd^c u \wedge T),
	\]
	where $T= \sum_{k=0}^{n-1} (dd^c v)^k \wedge (dd^c u)^{n-k-1}$ is a closed positive current. Let us point out that the identity
 $$ (dd^c v)^n- (dd^c u)^n= dd^c(v-u)\wedge T$$ is clear when $u,v$ are smooth. The general case is proven by approximating $u$ and $v$ by smooth decreasing sequences of psh functions and passing to the limit using Theorem \ref{thm: BT convergence}.
 
 By definition and since $\chi$ is smooth,
	\[
	\int_{\Omega} \chi ( dd^c v \wedge T ) = \int_{\Omega} \chi dd^c (v T ) = \int_{\Omega}  dd^c \chi  \wedge vT.
	\]
	Applying this for $u$ and computing the difference we obtain 
	\[
	\int_{\Omega} \chi (dd^c v \wedge T -dd^c u \wedge T) = \int_{\Omega} (v-u) dd^c \chi \wedge T=0,
	\]
	because $(v-u)=0$ in $\Omega \setminus K$ and $\chi\equiv 1$ in a neighborhood of $K$. Letting $\chi \to {\bf 1}_{\Omega}$ (the function that is identically equal to $1$ in $\Omega$) we conclude.
\end{proof}

\begin{exercise}
    Let $u<0$ be a bounded psh function in $\Omega$ which vanishes on the boundary $\partial \Omega$, and set $v= -(-u)^{1/2}$. Prove that $v$ is psh in $\Omega$ and $\int_{\Omega}(dd^c v)^n=+\infty$. What happens if we consider $v_\alpha:=-(-u)^\alpha$, with $\alpha \in (0,1)$?
\end{exercise}

\begin{lemma}[Comparison of total masses]\label{lem: CP mass form}
	Assume $u,v$ are bounded psh functions in $\Omega$ such that $u = v$ on $\partial \Omega$ and $u\leq v$ in $\Omega$. Then
	\[
	\int_{\Omega} (dd^c u)^n \geq \int_{\Omega}(dd^c v)^n. 
	\]
\end{lemma}

\begin{proof}
	For $\varepsilon>0$, define $w=\max(u,v-\varepsilon)$. Then $w=u$ near the boundary of $\Omega$,  hence by the above lemma, we have 
	\[
	\int_{\Omega} (dd^c \max(u,v-\varepsilon))^n=\int_{\Omega} (dd^c w)^n = \int_{\Omega} (dd^c u)^n. 
	\]
	Letting $\varepsilon\to 0$, thanks to Theorem \ref{thm: BT convergence} and Lemma \ref{lem: semicontinuity} we infer that
 $$ \liminf_{\varepsilon \rightarrow 0 } \int_{\Omega} (dd^c \max(u,v-\varepsilon))^n \geq \int_{\Omega} (dd^c \max(u,v))^n = \int_{\Omega} (dd^c v)^n.$$
Combining the above we can conclude.
\end{proof}

\begin{corollary}\label{cor: maximal function}
    If $u\in \PSH(\Omega)\cap L^{\infty}$, $u= 0$ on $\partial \Omega$ and $(dd^c u)^n=0$, then $u=0$. 
\end{corollary}
\begin{proof}
   Fix a negative strictly psh function $\rho$ in $\Omega$ (one can take $\rho(z)=|z|^2-R^2$, where $R>{\rm diam}(\Omega)$). Assume by contradiction that  $u<0$ in $\Omega$.  Then for some $\varepsilon>0$, the function $v=\max(u,\varepsilon \rho)$ satisfies $v\geq u$, $v=u=0$ on $\partial \Omega$, $v=\varepsilon \rho$ in an open subset of $\Omega$. From this we infer that $\int_\Omega (dd^c v)^n > 0$. Lemma \ref{lem: CP mass form} gives a contradiction.
\end{proof}

\subsection{Convergence in capacity}
A Borel set $E\subset \Omega$ is pluripolar if for each $z_0\in E$, there exist a neighborhood $U$ of $z_0$ and a function $u\in \PSH(U)$ not identically $-\infty$, such that $E\cap U \subset \{u=-\infty\}$. 

The Monge-Amp\`ere capacity of a Borel set $E$ in $\Omega$ is defined as 
\[
\capa(E):= \sup \left\{\int_E (dd^c u)^n \; : \; u\in \PSH(\Omega), -1\leq u \leq 0\right\}. 
\]
The outer capacity is defined as 
\[
\capa^*(E):= \inf\left\{\capa (U)\; : \; U \; \text{is open and} \; E\subset U\subset \Omega\right\}. 
\]
\begin{theorem}[Josefson's theorem]
	A Borel set $E$ is pluripolar if and only if $\capa^*(E)=0$ if and only if there exists $u\in \PSH(\Omega)$ such that $E\subset \{u=-\infty\}$. 
\end{theorem}

A sequence of Borel functions $f_j$ converges in capacity to a Borel function $f$ in $\Omega$ if, for any $\varepsilon>0$ and all compact sets $K \subset \Omega$
$$ \lim_{j\rightarrow +\infty} \capa^*( K\cap \{|f_j-f|\geq  \varepsilon\} )=0.$$
An important fact we will use several times in what follows is that monotone convergence implies convergence in capacity \cite[Proposition 4.25]{GZbook}.

\subsection{Plurisubharmonic envelopes}

Given a bounded function $h: \Omega \rightarrow \mathbb R$, we define the psh envelope $P(h)$ to be the largest psh function in $\Omega$ lying below $h$. More precisely, $P(h)$ is the upper semicontinuous regularization of 
\[
x\mapsto \sup\{u(x) \; : \; u \in \PSH(\Omega), \; u\leq h \; \text{in}\; \Omega \}. 
\]

Since $h$ is bounded $P(h)\neq -\infty$ and $P(h) \in \PSH (\Omega)$. In general, $P(h)\leq h$ quasi everywhere (i.e outside a pluripolar set) 
%as the upper semicontinuous regularization only changes the function away from a pluripolar set. 
When $h$ is continuous there is no need to take the upper semicontinuous regularization and $P(h)\leq h$ everywhere. 
\begin{lemma}\label{qe_envelope}
	If $h=g$ quasi everywhere in $\Omega$ then $P(h)=P(g)$. 
\end{lemma}

\begin{proof}
	Assume $f=g$ in $\Omega\setminus E$ where $E$ is a pluripolar set. By Josefson's theorem there exists $\psi\in \PSH^-(\Omega)$ such that $E\subset \{\psi=-\infty\}$. We then observe that for all $\varepsilon>0$, $P(f)+\varepsilon \psi \leq f$ and  $P(g)+\varepsilon \psi \leq g$ quasi everywhere (say outside $P$) and so $P(f)+\varepsilon \psi \leq g$ and $P(g)+\varepsilon \psi \leq f$ outside $E\cup P$. It follows that outside $P$ 
	\[
	P(f) + \varepsilon \psi \leq P(g) \leq P(f) - \varepsilon \psi. 
	\]
	Letting $\varepsilon\to 0$, we obtain $P(f)=P(g)$ quasi-everywhere, hence everywhere by the properties of the psh functions.
\end{proof}

\begin{lemma}[Balayage]\label{lem: balayage}
	Let $\varphi\in \PSH(\Omega)\cap L^{\infty}$ and $B(x_0,r)\Subset \Omega$. Then there exists $u\in \PSH(\Omega)$ such that $u\geq \varphi$ with equality in $\Omega\setminus B(x_0,r)$ and $(dd^c u)^n=0$ in $B(x_0,r)$.  
\end{lemma}
\begin{proof}
	Let $(\varphi_j)$ be a sequence of smooth psh functions decreasing to $\varphi$ in $\Omega'\Subset \Omega$, where $B=B(x_0,r)\Subset \Omega'$. For each $j$, we solve the Dirichlet problem \cite[Theorem 5.14]{GZbook} $(dd^c v_j)^n =0$ in $B$, with $v_j=\varphi_j$ on $\partial B$. The solution $v_j$ is the upper envelope of all psh functions in $B$ with boundary values $\varphi_j$, thus  $v_j\geq \varphi_j$ in $B$. We then define the function $u_j$ equal to $v_j$ in $B$ and equal to $\varphi_j$ in $\Omega\setminus B$. By construction $u_j$ is psh in $\Omega$. By the comparison principle again, $u_j$ decreases to some psh function $u$. We have $u\geq \varphi$ with equality in $\Omega\setminus B$, and $(dd^c u)^n=0$ in $B$ by Bedford-Taylor's convergence theorem (Theorem \ref{thm: BT convergence}). 
\end{proof}

\begin{theorem}[Orthogonal relation]\label{thm: orthog relation}
	Assume $h$ is quasi lower semicontinous and bounded in $\Omega$. Then the Monge-Amp\`ere measure $(dd^c P(h))^n$ vanishes in $\{P(h)<h\}$. 
\end{theorem}
A function $h$ is called quasi lower semicontinous if for each $\varepsilon>0$, there exists an open set $U$ such that $\capa(U) <\varepsilon$ and the restriction of $h$ on $\Omega\setminus  U$ is lower semicontinuous. 
\begin{proof}
	We first assume that $h$ is lower semicontinuous in $\Omega$. Then the set $U:=\{P(h)<h\}$ is open in $\Omega$, which allows us to use the balayage argument. Fix $x_0\in U$ and $\varepsilon>0$ such that $P(h)(x_0)<h(x_0)-3\varepsilon$. Since $P(h)$ is upper semicontinuous and $h$ is lower semicontinuous, we can find $r>0$ such that, for all $x\in \bar{B}$,
	\[
	P(h)(x)<P(h)(x_0)+ \varepsilon \; \text{and}\; h(x)>h(x_0)-\varepsilon,
	\] 
	where $B=B(x_0,r)$ is the Euclidean ball centered at $x_0$ of radius $r$. By Lemma \ref{lem: balayage}, we can find a function $u\in \PSH(\Omega)$ such that $u\geq P(h)$ with equality in $\Omega\setminus B$, and $(dd^c u)^n=0$ in $B$. By the maximum principle, for all $x\in B$,
	\[
	u(x)\leq \sup_{\partial B} u= \sup_{\partial B} P(h) < P(h)(x_0)+ \varepsilon < h(x_0) -2\varepsilon <h(x)-\varepsilon. 
	\]
	It thus follows that $u\leq h$ in $\Omega$ (recall that $u=P(h)$ outside $B$), hence by definition of $P(h)$, we have $u\leq P(h)$. By the construction of $u$, we also have $P(h)\leq u$, therefore $u=P(h)$. Consequently, $(dd^c P(h))^n=0$ in $B(x_0,r)$, hence it vanishes in $U$ since $x_0$ was chosen arbitrarily. 
	
	Assume now that $h$ is quasi lower semicontinuous. By \cite[Lemma 2.4]{GLZ19} there exists a decreasing sequence $(h_j)$ of lower semicontinuos functions which converges to $h$ in capacity and quasi everywhere. Thus $h_j \searrow g$ and $g=h$ except for a set $E$ of capacity zero. Moreover, \cite[Proposition 2.2]{GLZ19} insures that $P(h_j)$ decreases to $P(g)$, which is equal to $P(h)$ by Lemma \ref{qe_envelope}. By the previous step we know that $(dd^c P(h_j))^n=0$ on $\{P(h_j)<h_j\}$, or in other words
 $$\int_\Omega \left(1-e^{P(h_j)-h_j}\right)(dd^c P(h_j))^n=0.$$ Also observe that, since $P(h_j)$ is uniformly bounded, there exists $A>0$ such that $P(h_j)\geq -A$, hence for any Borel set $B$ we have 
 $${\bf 1}_B (dd^c P(h_j))^n \leq A^n \capa (B).$$

We then have
\begin{eqnarray*}
&&\int_\Omega  \left|1-e^{P(h_j)-h}\right|(dd^c P(h_j))^n \\
&&=\int_{\Omega\setminus E} \left(1-e^{P(h_j)-g}\right)(dd^c P(h_j))^n +\int_{E} \left(1-e^{P(h_j)-h}\right)(dd^c P(h_j))^n\\
&&\leq  \int_{\Omega\setminus E} \left(1-e^{P(h_j)-h_j}\right)(dd^c P(h_j))^n + A^n \sup_\Omega \left|1-e^{P(h_j)-h}\right| \capa (E)\\
&& = 0.
\end{eqnarray*}
It follows from Theorem \ref{thm: BT convergence} that
$$\int_\Omega  \left|1-e^{P(h)-h}\right|(dd^c P(h))^n=0.$$
We can then conclude that the measure $dd^c P(h))^n$ is supported on the contact set $\{P(h)=h\}$
\end{proof}

\subsection{Monge-Amp\`ere operator with non-closed forms}
Let $(X,\omega)$ be a compact Hermitian manifold of dimension $n$ and fix a smooth real $(1,1)$-form $\theta$.    We want to define the Monge-Amp\`ere measure $(\theta+dd^c u)^n$ for a bounded $\theta$-psh function $u$. Let $\Omega$ be an open subset biholomorphic to an Euclidean ball in $\mathbb{C}^n$. Then there exists a smooth strictly psh function $\psi$ in a neighborhood of $\Omega$. 

\begin{definition}
	A function $u$ is quasi-psh in an open set $U\subset X$ if it is locally the sum of a smooth and a psh function. It is $\theta$-psh if it is quasi-psh and $\theta+dd^c u\geq 0$ in the sense of currents on $X$. We let $\PSH(U,\theta)$ be the set of all $\theta$-psh functions in $U$ which are locally integrable. 
\end{definition}

\begin{definition}
    Assume $u$ is quasi-psh and bounded in $\Omega$, and let $\psi$ be a smooth psh function in $\Omega$ such that $dd^c (\psi + u)\geq 0$. For $1\leq p\leq n$, we define 
\begin{equation}\label{def ddc}
(dd^c u)^p = \sum_{k=0}^p (-1)^k\binom{p}{k} (dd^c \psi)^k \wedge (dd^c (\psi+u))^{p-k}.
\end{equation}
\end{definition}
Since $\psi+u$ is psh, $(dd^c (\psi+u))^{p-k}$ is a closed positive $(p-k,p-k)$-current. This, combined with the fact that $dd^c \psi$ is a smooth form, ensures that $(dd^c u)^p$ (defined as above) is a $(p,p)$-current of order zero. We now show that it does not depend on the choice of $\psi$. Assume $\psi'$ is another smooth psh function in $\Omega$. By replacing $\psi'$ with $\psi+\psi'$, we can further assume that $dd^c \psi'\geq dd^c \psi$. 
Taking convolution with smoothing kernel, $(\psi+u)_j:= (\psi+u)\star \rho_{1/j}$, we obtain a sequence of smooth psh functions decreasing to $\psi+u$. Then 
\[
\sum_{k=0}^p \binom{p}{k} (-dd^c \psi)^k \wedge (dd^c (\psi+u)_j)^{p-k}=\sum_{k=0}^p \binom{p}{k} (-dd^c \psi')^k \wedge (dd^c (\psi+u)_j+\psi'-\psi)^{p-k}.
\]
Observe that the above identity follows by the formal trick $(-dd^c \psi)^k= (-dd^c \psi' + dd^c(\psi'-\psi))^k$, re-arranging the terms and using that for $\ell=0, \cdots k$
$$\binom{p}{k} \binom{k}{\ell} = \binom{p}{k-\ell} \binom{p-k-\ell}{\ell}.$$
Also, all the computations make sense since all the forms are smooth.
By Bedford-Taylor's convergence theorem, we obtain 
\[
\sum_{k=0}^p \binom{p}{k} (-dd^c \psi)^k \wedge (dd^c (\psi+u))^{p-k}=\sum_{k=0}^p \binom{p}{k} (-dd^c \psi')^k \wedge (dd^c (\psi'+u))^{p-k}.
\]

Given a bounded quasi-psh function $u$ on $X$, by the above analysis, the current $(dd^c u)^p$ is well-defined on each coordinate chart and, by the above observation, can be glued to define a global closed $(p,p)$-current on $X$.

\begin{lemma}
    Let $u$ be a bounded quasi-psh function on $X$ and $u_j$ be a sequence of smooth quasi-psh functions decreasing to $u$. Assume also that, for some fixed constant $A>0$, we have $A\omega+dd^c u_j\geq 0$ for all $j$. Then $(dd^c u_j)^p$ weakly converges to $(dd^c u)^p$. In particular, if $u \in \PSH(X,\theta)$, then $\theta_u^p:=(\theta+dd^c u)^p$ is a positive $(p,p)$-current. 
\end{lemma}

Note that $\theta_u^p$ is defined as $\sum_{k=0}^p  \binom{p}{k}\theta^{p-k} \wedge (dd^c u)^{k}$, where the last term is defined in \eqref{def ddc}.
\begin{proof}
Since the statement is local, we can assume that $X=\Omega$ is a coordinate chart. 
    Assume $\psi$ is a smooth psh function such that $dd^c \psi\geq A\omega$ in $\Omega$. By Bedford-Taylor's convergence theorem, the following weak convergence holds
    \[
    \sum_{k=0}^p \binom{p}{k}(-dd^c \psi)^k \wedge (dd^c (\psi+u_j))^{p-k} \to \sum_{k=0}^p \binom{p}{k}(-dd^c \psi)^k \wedge (dd^c (\psi+u))^{p-k}. 
    \]
    Observe that the assumption $A\omega+dd^c u_j \geq 0$ is used to ensure that $(dd^c (\psi+u_j))^{p-k} $ and $(dd^c (\psi+u))^{p-k} $ are positive currents.
    This proves the desired convergence. If $u\in \PSH(X,\theta)$, then, by taking convolution with a smoothing kernel, we can locally find a sequence of smooth functions $u_j$ decreasing to $u$ such that $\theta +2^{-j}\omega +dd^c u_j\geq 0$. Then the convergence result in the first statement yields that 
    \[
    (\theta +2^{-j}\omega +dd^c u_j)^p \to (\theta+dd^c u)^p,
    \]
    in the weak sense of currents, proving that the limit is also a positive current. 
\end{proof}

\begin{lemma}\label{lem: max principle global}
	Assume $u,v$ are bounded $\theta$-psh functions in an open set of $\mathbb C^n$. Then 
	\[
	(\theta+dd^c \max(u,v))^n \geq {\bf 1}_{\{u\geq v\}} (\theta+dd^c u)^n +  {\bf 1}_{\{u< v\}} (\theta+dd^c v)^n. 
	\]
 In particular, if $u\leq v$ then ${\bf 1}_{\{u= v\}} (\theta+dd^c u)^n \leq {\bf 1}_{\{u= v\}} (\theta+dd^c v)^n$. 
\end{lemma}

\begin{proof}
	The proof is similar to that of Corollary \ref{cor: max principle} using plurifine locality (see also \cite[Lemma 1.2]{GL22}). 
\end{proof}

\begin{lemma}\label{lem: CP ball}
    Let $B=B(x_0,r)$ be an Euclidean ball in $\mathbb{C}^n$, $r>0$. Assume $u,v$ are bounded $\theta$-psh functions in $B$  and 
    \[
 (\theta+dd^c u)^n \leq (\theta+dd^c v)^n.
    \]
    Then 
    \[
    u -v \geq \liminf_{z\to \partial B} (u-v)(z). 
    \]
\end{lemma}
\begin{proof}
   By adding a constant we can assume $\liminf_{z\to \partial B} (u-v)(z)=0$, meaning that for any $\varepsilon>0$, there exists a sequence $(z_j) \subseteq B$, $z_j\rightarrow z\in \partial B$ such that $u(z_j)\geq v(z_j)-\varepsilon$. Then, we consider $v_{\varepsilon}=\max(u,v-\varepsilon)$, which is $\theta$-psh in $B$ and satisfies $v_{\varepsilon}=u$ in a neighborhood of $\partial B$. By the maximum principle (Corollary \ref{cor: max principle}) together with the assumption, $(\theta+dd^c v_{\varepsilon})^n \geq (\theta+dd^c u)^n$. Our goal is to prove that $u\geq v_{\varepsilon}$ and then let $\varepsilon\to 0$. In the sequel, we remove the subscript $\varepsilon$. Since $u,v$ are bounded, for some large constant $A$, we have $(u-v)(z)\geq A (|z-x_0|^2-r^2)$ for any $z\in B$; in particular $w=P(u-v)\in \PSH(B)$ vanishes on $\partial B$, since $A (|z-x_0|^2-r^2)\leq w\leq u-v$. We next show that $(dd^c w)^n=0$. Since $(dd^c w)^n$ is supported on the contact set  $D=\{w=u-v\}$ (Theorem \ref{lem: balayage} can be applied since $u-v$ is quasi l.s.c.), it is enough to show that $\idd_D (dd^c w)^n=0$. This follows from the maximum principle Lemma \ref{lem: max principle global} as we show below. Since $w+v$ and $u$ are bounded $\theta$-psh functions with $w+v\leq u$, we have, by Lemma \ref{lem: max principle global}, 
   \[
   \idd_D (\theta+dd^c (w+v))^n \leq \idd_D (\theta+dd^c u)^n. 
   \]
   Expanding the left-hand side and using the assumption again we obtain 
    \[
   \idd_D (\theta+dd^c v )^n + \idd_D (dd^c w)^n \leq  \idd_D (\theta+dd^c (w+v))^n  \leq \idd_D (\theta+dd^c u)^n\leq  \idd_D (\theta+dd^c v )^n.
   \]
	We then obtain $\idd_D(dd^c w)^n=0$, as we wanted. By Corollary \ref{cor: maximal function} we infer that $w=0$. In particular $u-v\geq 0$, finishing the proof. 
\end{proof}

\begin{prop}\label{prop: CP}
	  Let $U\subset X$ an open set and assume there exists a bounded smooth strictly psh function $\rho$ in $U$.  Assume $u,v \in \PSH(U,\theta)\cap L^{\infty}$ such that $\theta_u^n \leq \theta_v^n$ in $U$.  Then 
   \[
   u-v \geq \liminf_{z\to \partial U} (u-v)(z).
   \]
\end{prop}

\begin{proof}
For $a>0$ we define 
\[
m_a = \liminf_{z\to \partial U} (u(z)-v(z) - a \rho(z)). 
\]
We first prove that $u(x)-v(x) - a \rho(x) \geq m_a$ for all $x\in U$. Clearly $$(u-v)(x) -a\rho(x) \geq \inf_{x\in U}(u(x)-v(x)-a\rho(x)).$$  Now, take the sequence $(x_j)\subset U$ such that 
\[
\lim_{j\rightarrow +\infty} (u-v)(x_j) -a\rho(x_j)=\inf_{x\in U}(u(x)-v(x)-a\rho(x)).
\]
A subsequence of $(x_j)$ converges to $x_0\in \overline{U}$.
If the desired inequality fails then $x_0\in U$. This is our contradiction assumption.
Take a small ball $B=B(x_0,r)\subset U$ around $x_0$, and consider 
\[
\varphi(x) := v(x) + a\rho(x) - a |x-x_0|^2, \; x\in B. 
\]
Here, without loss of generality, we assume that $dd^c |x|^2\leq dd^c \rho$.
Using the assumption and a direct computation (in particular that $dd^c (\rho- |x-x_0|^2) \geq 0$) we obtain 
\[
(\theta+dd^c u)^n \leq (\theta+dd^c v)^n \leq (\theta+dd^c \varphi)^n, 
\]
in $B$. 
The comparison principle, Lemma \ref{lem: CP ball}, then yields, for all $x\in B$,
\[
(u-\varphi)(x) \geq \liminf_{z\to \partial B} (u-\varphi)(z) \geq \inf_{x\in U} (u(x)-v(x) -a\rho(x)) + ar^2.   
\]
Taking $x=x_j$ and letting $j\to +\infty$ (remember that $x_j\rightarrow x_0$), we arrive at 
\[
 \lim_{j\to +\infty}(u(x_j)-v(x_j) -a\rho(x_j)) \geq \inf_{x\in U} (u(x)-v(x) -a\rho(x)) + ar^2,
\]
which yields $ar^2\leq 0$, a contradiction. 
We thus conclude that 
\[
u(z)-v(z)-a\rho(z)\geq \liminf_{y\to \partial U} (u(y)-v(y)-a\rho(y)).
\]
Since $\rho$ is bounded, letting $a\to 0$, we arrive at the conclusion. 
\end{proof}

\begin{theorem}[Domination principle]\label{thm: domination principle}
    Assume $u,v \in \PSH(X,\omega)\cap L^{\infty}$ satisfy
    \[
    \idd_{\{u<v\}}(\omega +dd^c u)^n\leq c \idd_{\{u<v\}} (\omega +dd^c v)^n,
    \]
    for some constant $0\leq c<1$. Then $u\geq v$. 
\end{theorem}

\begin{proof}
  Without loss of generality, we assume $v\geq 1$. We fix a constant $a$ such that $c<a^n<1$, and we prove that $u\geq av$; the result then follows by letting $a\to 1$. Assume it is not the case and let $(x_j)$ be a sequence converging to $x_0\in X$ with 
  \[
  \lim_{j\to +\infty} (u(x_j) - av(x_j)) = m_a:= \inf_X (u-av) <0. 
  \]
  Let $B$ be a small neighborhood of $x_0$ and fix a smooth  strictly psh function $\rho$ such that   $\rho=0$ on $\partial B$, $|\rho|\leq v$, and  $dd^c \rho \leq \omega$. One can take a holomorphic coordinate chart around $x_0$ which is biholomorphic to the unit ball and $\rho(x) =b(|x|^2-1)$, for a small constant $b$.  Observe that $\rho<0$ in $B$. Consider now $\varphi=av- (1-a) \rho$. A direct computation then shows that  $\omega+dd^c \varphi\geq 0$ in $B$. Also, by construction $\varphi \leq v$ (because $-\rho \leq v$) and so $\{u<\varphi\}\cap B \subset \{u<av-(1-a)\rho \}\subset \{u<v\}$. Hence by assumption and using the fact that $\omega+dd^c (-\rho) \geq 0$, we get
  \[
  \idd_{\{u<\varphi\}}(\omega+dd^c u)^n\leq    c\idd_{\{u<\varphi\}}(\omega+dd^c v)^n \leq a^n\idd_{\{u<\varphi\}}(\omega+dd^c v)^n \leq \idd_{\{u<\varphi\}}(\omega+dd^c \varphi)^n.
  \]
From this and the maximum principle (Corollary \ref{cor: max principle}) we infer 
   \[
(\omega+dd^c u)^n\leq (\omega+dd^c  \max(u,\varphi))^n\; \; \text{in}\; B.
  \]
  It thus follows from the comparison principle, Proposition \ref{prop: CP}, that 
 \[
 \inf_B(u- \max(u,\varphi)) \geq \liminf_{x\to \partial B} (u-\max(u,\varphi))(x). 
 \]
 Evaluating this inequality at $x_j$ and letting $j\to +\infty$, we arrive at 
 \begin{equation}\label{contra}
 \lim_{j\to +\infty} \min(u(x_j) -av(x_j) +(1-a)\rho(x_j),0) \geq \liminf_{x\to \partial B} \min(u(x)-av(x),0). 
 \end{equation}
 Since $\liminf_{x\to \partial B} (u(x)-av(x))\geq m_a$, and $m_a<0$, we infer that
 \[
 \liminf_{x\to \partial B} \min(u(x)-av(x),0)\geq m_a. 
 \]
Using \eqref{contra}, we conclude that $\rho(x_0) \geq 0$, that is a contradiction. 
\end{proof}

The usual comparison principle in the K\"ahler setting states that 
\begin{equation}\label{eq: usual CP}
\int_{\{u<v\}} (\omega+dd^c v)^n \leq \int_{\{u<v\}} (\omega+dd^c u)^n, 	\; u,v \in \PSH(X,\omega) \cap L^{\infty}
\end{equation}
Taking $u=-C$ and then $v=C$, for sufficiently large $C$, we see that the validity of \eqref{eq: usual CP} implies the invariance of the Monge-Amp\`ere mass 
\begin{equation}
	\label{eq: invariance mass}
	\int_X(\omega+dd^c u)^n = \int_X \omega^n, \; u\in \PSH(X,\omega)\cap L^{\infty}. 
\end{equation} 
As shown in \cite{Chi24}, the validity of \eqref{eq: invariance mass} is equivalent to the Guan-Li condition $dd^c \omega=0$ and $dd^c (\omega^2)=0$. Conversely, the validity of \eqref{eq: invariance mass} implies that of \eqref{eq: usual CP} as we now explain. 
By the maximum principle we have 
\begin{flalign*}
	{\bf 1}_{\{u\geq v\}}(\omega+dd^c \max(u,v))^n & \geq {\bf 1}_{\{u\geq v\}}(\omega+dd^c u)^n\\
	{\bf 1}_{\{u<v\}}(\omega+dd^c \max(u,v))^n &= {\bf 1}_{\{u<v\}}(\omega+dd^c v)^n.
\end{flalign*}
Integrating over $X$ and summing up we obtain 
\begin{flalign*}
	\int_X (\omega+dd^c \max(u,v))^n &\geq \int_X (\omega+dd^c u)^n \\
	&+\int_{\{u<v\}}(\omega+dd^c v)^n-\int_{\{u<v\}}(\omega+dd^c u)^n.
\end{flalign*}
Using $\int_X (\omega+dd^c \max(u,v))^n = \int_X (\omega+dd^c u)^n$, we arrive at \eqref{eq: usual CP}. From \eqref{eq: usual CP} we deduce the domination principle as follows. Without loss of generality, we assume $v\geq 1$. We fix a constant $a$ such that $c<a^n<1$, and we prove that $u\geq av$; the result then follows by letting $a\to 1$. By the comparison principle \eqref{eq: usual CP} we have 
\begin{flalign*}
	\int_{\{u<av\}} (\omega+dd^c av)^n &\leq \int_{\{u<av\}} (\omega+dd^c u)^n\\
	&\leq c \int_{\{u<av\}} (\omega+dd^c v)^n \\
	&\leq ca^{-n}\int_{\{u<av\}} (\omega+dd^c av)^n.
\end{flalign*}
Since $ca^{-n}<1$, we infer that all the terms above are zero. Expanding the last term we obtain 
\[
0 = \int_{\{u<av\}} ((1-a) \omega + a(\omega+dd^c v))^n \geq \int_{\{u<av\}} (1-a)^n \omega^n, 
\]
hence the set $\{u<av\}$ has zero Lebesgue measure, and it must be empty because $u$ and $v$ are quasi-psh. 

In the setting of Theorem \ref{thm: domination principle}, we do not assume $d\omega=0$ and the above proof via the usual comparison principle does not work.

As a direct consequence of the domination principle we have 
\begin{corollary}\label{cor:MAconstant a}
    Assume $u,v$ are bounded $\omega$-psh functions on $X$ such that 
    $$
    (\omega+dd^c u)^n \leq c (\omega+dd^c v)^n, 
    $$
    for some positive constant $c$. Then $c\geq 1$. 
\end{corollary}
\begin{proof}
By contradiction assume $c<1$. Theorem \ref{thm: domination principle} (applied to $u$ and $v+C$, $C>0$) ensures that $u\geq v+C$, for any $C>0$, that is not possible.
\end{proof}

\begin{corollary}\label{cor:MAexp a}
    Assume $u,v$ are bounded $\omega$-psh functions on $X$ such that 
    $$
    (\omega+dd^c u)^n \leq e^{\lambda(u-v)} (\omega+dd^c v)^n, 
    $$
    for some positive constant $\lambda>0$. Then $u\geq v$. 
\end{corollary}

\begin{proof} Let $a>0$. By assumption 
$$\idd_{\{u<v-a\}}  (\omega+dd^c u)^n \leq \idd_{\{u<v-a\}}  e^{\lambda(u-v)} (\omega+dd^c v)^n \leq \idd_{\{u<v-a\}}  e^{-a\lambda } (\omega+dd^c v)^n. $$
Theorem \ref{thm: domination principle} ensures that $u\geq v-a$, for any $a>0$. The conclusion follows letting $a\rightarrow 0$.
\end{proof}

\begin{corollary}\label{cor: dom princ}
    If $u\in \PSH(X,\omega)\cap L^{\infty}(X)$, then $\int_X (\omega+dd^c u)^n>0$. 
\end{corollary}

\begin{proof}
If $\int_X(\omega+dd^c u)^n=0$ then $(\omega+dd^c u)^n=0$ and the domination principle, applied for any $v\in \PSH(X,\omega)\cap L^{\infty}$, gives $u\geq v$. In particular, $u\geq C$ for all constant $C$. This is a contradiction.  
\end{proof}

%In fact, we have a uniform positive lower bound: 
%\begin{prop}\label{prop: mass uniform bound from below}
    %For each $C>0$, there exists a constant $m>0$ such that, for all $u\in \PSH(X,\omega)$ with $-C\leq u\leq 0$, we have 
    %\[
    %\int_X (\omega+dd^c u)^n\geq m. 
   % \]
%\end{prop}
%\begin{proof}
    %We assume by contradiction that there exists a sequence $-C\leq u_j\leq 0$ in $\PSH(X,\omega)$ such that 
    %\[
    %\int_X (\omega+dd^c u_j)^n\leq 2^{-j}.
    %\]
    %Define 
    %\[
    %v_{j,k}= P_{\omega}(\min_{j\leq l\leq k} u_l).
    %\]
    %Then $v_{j,k}\searrow v_j$ as $k\to +\infty$ and $-C\leq v_j\leq 0$. By the maximum principle 
    %\[
    %\int_X (\omega+dd^c v_{j,k})^n \leq \sum_{l=j}^{k} 2^{-l} \leq 2^{-j+1},
    %\]
    %thus, letting $k\to +\infty$, we obtain $\int_X (\omega+dd^c v_j)^n \leq 2^{-j+1}$. Since $v_j\nearrow v$ and $v\in \PSH(X,\omega)$, with $-C\leq v\leq 0$, we obtain by Bedford-Taylor convergence theorem that $\int_X (\omega+dd^c v)^n=0$, a contradiction. 
%\end{proof}

%\section{Unbounded plurisubhamonic functions}

\section{$L^{\infty}$-estimate in an Euclidean domain}
We let $\mathcal T$ denote the set of bounded psh functions in $\Omega$ vanishing on the boundary, and $\mathcal T_0$ those that have Monge-Ampère mass less than $1$: 
\[
\mathcal T_0:= \left\{u\in \PSH(\Omega), \; u =0 \; \text{on}\; \partial \Omega,\; \int_{\Omega}(dd^c u)^n \leq 1\right\}. 
\] 

\begin{theorem}\label{thm: C0 estimate}
	Assume $\varphi$ is a bounded psh function in $\Omega$ such that $\varphi=0$ on $\partial \Omega$, $(dd^c \varphi)^n =\mu$, where $\mu$ is a positive measure satisfying $\mu(\Omega)\leq 1$ and 
	\[
	A_m(\mu):= \sup\left \{\int_{\Omega} |u|^m d\mu\; :\; u\in \mathcal{T}_0 \right\}<+\infty,
	\]
	for some $m>n$. 
	Then 
	\[
	|\varphi| \leq C, 
	\] 
	where $C$ is a uniform constant depending on $A_m(\mu)$, $n,m$. 
\end{theorem} 

\begin{proof}
In the proof below,  to simplify the notation, we use $C$ to denote various different uniform positive constants. We define 
\[
T:= \sup \{t\geq 0 \; : \; \mu(\varphi<-t) >0\}. 
\]
If $t>T$ then $\mu (\varphi<-t)= (dd^c \varphi)^n(\varphi<-t)=0$ and the domination principle (Theorem \ref{thm: domination principle}) gives $\varphi\geq -t$, hence $\varphi\geq -T$ in $\Omega$. 

We fix $T_0\in (0,T)$, and define $g:[0,+\infty) \rightarrow [0,+\infty)$ by $g(0)=1$ and 
\[
g'(t):= \begin{cases}
		\frac{1}{(1+t)^{3/2}\mu(\varphi<-t)}, \; t\leq T_0,\\
		\frac{1}{(1+t)^{3/2}}, \; t>T_0. 
		\end{cases}
\]
Then $g$ is an increasing function which is Lipschitz. 
We next define $\chi: [0,+\infty) \rightarrow [0,+\infty)$ by $\chi(0)=0$ and $\chi'(t) = g(t)^{1/n}$. Then  $\chi$ is a $C^1$, convex, increasing function with $\chi(0)=0$, $\chi'(0)=1$. In particular (assuming we can compute the second derivative) $g'(t)=n\chi''(t)\chi'(t)^{(n-1)} $.
Observe that $\chi(1)=\chi(0)+ \chi'(s)$, for some $s\in [0,1]$, which yields $\chi(1)\geq 1$ since $\chi'(s)\geq \chi'(0)=1$.

Let $u:= P(-\chi(-\varphi))$ be the largest psh function in $\Omega$ lying below $-\chi(-\varphi)$. Again using the intermediate value theorem, we infer that for each $t>0$, there exists some $0<s<t$ such that $\chi(t)-\chi(0)=\chi(t)=t\chi'(s)$, and the latter term is less that $t\chi'(t)$ because $\chi'$ is increasing. 

Observe that by the maximum principle, $\varphi \leq 0$, hence for $t=-\varphi(x)$ we have $\chi(-\varphi)\leq (-\varphi)\chi'(M)$, where $M=\sup_{\Omega}|\varphi|$. It follows that $-\chi(-\varphi)\geq \chi'(M) \varphi$, hence $u=P(-\chi(-\varphi))\geq \chi'(M)\varphi$. Consequently $u$ is a bounded psh function with zero boundary data by construction (since $\chi(0)=0$), i.e. $u\in \mathcal T$. 

We next estimate the total mass of $(dd^c u)^n$.  Let $\gamma: [0,+\infty) \rightarrow [0,+\infty)$ be the inverse of $\chi$. In particular $\gamma$ is an increasing function. A direct computation (and an approximation argument needed to compute the second derivatives of $\gamma$) shows that 
\[
dd^c (-\gamma(-u)) = \gamma'(-u) dd^c u - \gamma''(-u) du \wedge d^c u \geq \gamma'(-u) dd^c u,
\] 
where in the last inequality we used that $\gamma$ is concave.
Hence $-\gamma(-u)$ is psh in $\Omega$ and $(dd^c -\gamma(-u))^n \geq \gamma'(u)^n (dd^c u)^n$ in $\Omega$. By the orthogonal relation Theorem \ref{thm: orthog relation} ($-\chi(-\varphi)$ is quasi continuous) we have 
\[
(dd^c u)^n = {\bf 1}_{\{u=-\chi(-\varphi)\}} (dd^c u)^n.
\]
We also have $u\leq -\chi(-\varphi)$, hence $-\gamma(-u) \leq \varphi$ with equality on the contact set $D:=\{u=-\chi(-\varphi)\}$. The maximum principle (Corollary \ref{cor: max principle}) thus gives 
\[
 {\bf 1}_{D} (dd^c (-\gamma(-u)))^n \leq  {\bf 1}_{D} (dd^c \varphi)^n. 
\]
Using that on $\{u=-\chi(-\varphi) \}$ we have $(\gamma'(-u))^{-1}=\chi'(-\varphi)$, we then get 
\[
(dd^c u)^n  \leq  {\bf 1}_{D} (\gamma'(-u))^{-n} dd^c (-\gamma(-u))^n  \leq  {\bf 1}_{D} (\chi'(-\varphi))^n (dd^c \varphi)^n \leq g(-\varphi) \mu. 
\]
By Lebesgue integral formula (using the variable change $t=g(s))$ and the definition of $g$ we arrive at
\begin{flalign*}
	\int_{\Omega} g(-\varphi) d\mu &= \int_0^{+\infty} \mu(g(-\varphi)>t)dt \\
	&=\int_0^{g(0)} \mu(g(-\varphi)>t)dt + \int_{g(0)}^{+\infty} \mu(g(-\varphi)>t)dt \\
	& \leq  g(0)\mu(\Omega) + \int_0^{+\infty} \mu(\varphi<-t) g'(t)dt\\
	&\leq \mu(\Omega) + \int_0^{+\infty} (1+t)^{-3/2}dt \leq 3.
\end{flalign*}
We then infer that $3^{-1/n} u\in \mathcal T_0$. 
 By the assumption on $\mu$, it thus follows that $\int_{\Omega}|u|^m d\mu \leq 3^{m/n}A_m (\mu)\leq C$. Hence by Chebyshev's inequality and the fact that $\{ -\chi(-\varphi) <-\chi(t)\}\subseteq \{ u<-\chi(t)\}$, we obtain 
	\[
	\mu(\varphi<-t) \chi(t)^m \leq \mu( \{ u<-\chi(t)\}) \chi(t)^m \leq C,
	\]
	which gives, for $0<t<T_0$, 
	\[
	\chi(t)^m \leq C g'(t)(1+t)^{3/2}\leq  CT_0^{3/2}\chi''(t) \chi'(t)^{n-1}.
	\]
	Multiplying with $\chi'(t)$ and integrating on $[0,t]$, we obtain 
	\[
	\chi(t)^{m+1} \leq CT_0^{3/2}\chi'(t)^{n+1}. 
	\]
	 We then have 
	 \begin{equation}\label{ineq. bound}
	 1\leq C T_0^{\beta} \frac{\chi'(t)}{\chi(t)^{\alpha}}, 
\end{equation}
	 with $\alpha = (1+m)/(1+n)>1$ and $\beta=3/(2(n+1))<1$. Note that since $\chi(1)\geq 1$
  $$\int_1^{T_0} \frac{\chi'(t)}{\chi(t)^{\alpha}} \, dt= \frac{1}{\alpha-1}(\chi(1)^{1-\alpha} - \chi(T_0)^{1-\alpha} ) \leq \frac{1}{\alpha-1}\chi(1)^{1-\alpha} \leq  \frac{1}{\alpha-1}.$$
 Hence integrating on $[1,T_0]$ the inequality in \eqref{ineq. bound}, we obtain $T_0\leq CT_0^{\beta}$, i.e. $T_0^{1-\beta}\leq C$. Since this is true for all $T_0<T$, we obtain a uniform bound for $T$, hence for $\sup_{\Omega}|\varphi|$. 
\end{proof}

%In the proof above we have used the following.  If $\chi: [0,+\infty) \rightarrow [0,+\infty)$ is concave, increasing, then 
%\[
%(dd^c -\chi(-\varphi))^n \geq \chi'(-\varphi)^n (dd^c \varphi)^n. 
%\] 
%If $\varphi$ is smooth then the the inequality follows from a direct computation: 
%\[
%(dd^c -\chi(-\varphi)) = \chi'(-\varphi)dd^c \varphi -\chi''(-\varphi)d\varphi \wedge d^c \varphi \geq  \chi'(-\varphi)dd^c \varphi,  
%\]
%since $\chi''\leq 0$. The general case can be proved by an approximation argument. 

We next give examples of measures $\mu$ for which $A_m(\mu)$ is finite. 
\begin{lemma}\label{lem: compact support}
	Assume $\mu=fdV$ with $f\in L^p(\Omega)$, for some $p>1$, and $f$ has compact support in $\Omega$. Then $A_m(\mu)<+\infty$ for all $m> 0$. 
\end{lemma}
\begin{proof}
	Assume by contradiction that it is not the case. Then there exists a sequence $(u_j)$ in $\mathcal T_0$ such that $\int_{\Omega} |u_j|^m fdV\to +\infty$. Let $K$ be a non-pluripolar compact subset of $\Omega$ such that $f=0$ in $\Omega\setminus K$. By H\"older's inequality, we also have $\int_K |u_j|^{mq} dV \to +\infty$, where $q$ is the conjugate of $p$:  
	\[
	\int_K |u_j|^mfdV \leq  \|u_j\|_{mq}^m \|f\|_p. 
	\]
	By Hartogs' lemma, $u_j$ converges locally uniformly to $-\infty$, hence $M_j=\sup_K u_j \to -\infty$. It follows that, for $j$ big enough, $|M_j|^{-1}u_j\leq h_K^*$, where 
 \[
 h_K^*:= \big( \sup\{ v\quad  v\in \PSH(\Omega), \, v\leq 0 \; v|_K\leq -1\} \big)^*
 \]
 is the relative extremal function of $K$ in $\Omega$. Since $u_j=h_K^*=0$ on $\partial \Omega$, we can use the comparison principle of total masses Lemma \ref{lem: CP mass form} to deduce that $0<\int_{\Omega} (dd^c h_K^*)^n \leq \int_{\Omega} |M_j|^{-n} (dd^c u_j)^n \leq |M_j|^{-n}\to 0$, a contradiction. 
 %\textcolor{red}{how do we know that $0<\int_{\Omega} (dd^c h_K^*)^n$? capacity of K non pluripolar}
\end{proof}

\begin{theorem}\label{thm: Linfty estimate}
	Assume $(dd^c \varphi)^n=fdV$ with $0\leq f\in L^p(\Omega)$, $p>1$, and $\varphi\in \PSH(\Omega)$ with $\varphi|_{\partial \Omega}=0$. Then $|\varphi|\leq C$, for a uniform constant depending on $diam(\Omega)$, $p$, $\|f\|_p$. 
\end{theorem}
\begin{proof}
	Fix $R>0$ and $x_0\in \Omega$ such that $\Omega\subset B(x_0,R)$. We solve the Dirichlet problem $(dd^c u)^n =\frac{f}{\|f\|_{L^p}} \idd_{\Omega} dV$ in $B(x_0,2R)$ and $u=0$ on $\partial B(x_0,2R)$. By  Theorem \ref{thm: C0 estimate} and Lemma \ref{lem: compact support}, we have $|u|\leq C$, for a uniform constant $C$ (which depends on $\|f\|_{L^p}$). Now, by the maximum principle, $\varphi=0\geq u$ on $\partial \Omega$ and $(dd^c u)^n = (dd^c \varphi)^n$ in $\Omega$. The comparison principle (Proposition \ref{prop: CP}) gives $\varphi\geq u\geq -C$, yielding the result. 
\end{proof}

We observe that \cite[Theorem B]{ACKPZ09} and H\"older's inequality ensure that any positive measure $\mu=fdV$ with $L^p$ density, $p>1$, satisfies the assumption $A_{\mu}<+\infty$. The above result can then be obtained from Theorem \ref{thm: C0 estimate} and \cite[Theorem B]{ACKPZ09}. For pedagogical reasons we choose to present a more straightforward proof.

%\begin{corollary}
	%Assume $f\in L^p(\Omega)$ for some $p>1$. Then the psh function $\varphi$ solving $(dd^c \varphi)^n = fdV$, $\varphi=0$ on $\partial \Omega$ is unique and continuous. Moreover, there is a uniform constant $C$ such that 
	%\[
	%|\varphi| \leq C \|f\|_p^{1/n}. 
	%\]
%\end{corollary}
%\begin{proof}
	%Set $B= \|f\|_p^{-1/n}$ and $\mu=\|f\|_{p}^{-1}f dV$. Then $\mu(\Omega)\leq 1$ and $(dd^c B\varphi)^n=\mu$, hence $B|\varphi| \leq C$, yielding $|\varphi|\leq C \|f\|_p^{1/n}$.
	
	%Next, we approximate $f$ in $L^p$ with $f_j$ continuous. We solve the Dirichlet problem with continuous boundary data \cite[Theorem 5.16]{GZbook} and we let $\varphi_j\in C(\bar{\Omega})$ solution of $(dd^c \varphi_j)^n = f_j dV$. We have $|\varphi_j-\varphi_k| \leq C \|f_j-f_k\|_p^{1/n}$, hence $\varphi_j$ converges uniformly to $\varphi$, proving the continuity of $\varphi$. 
%\end{proof}

\section{A priori estimates on compact Hermitian manifolds}
We fix a hermitian form $\omega$ on a compact complex manifold $X$ of complex dimension $n$. We fix $p>1$ and we use $L^p(X)$ to denote the space $L^p(X,\omega^n)$. The $L^p$ norm of a measurable function $f$ is denoted by $\|f\|_p = \|f\|_{L^p(X)}$. 
\subsection{$L^{\infty}$-estimate}
The goal of this section is to prove an a priori $L^{\infty}$ estimate for solutions to Monge-Amp\`ere equations with right-hand side in $L^p$, for some $p>1$. We follow the strategy in \cite{GL23Crelle}, starting first with the local case. Having the bound for solutions in Euclidean balls we consider a double covering of the compact manifold $X$ and glue the solutions together to obtain a uniformly bounded subsolution. The domination pricniple then ensures a uniform bound for the solution. The main input here is the $L^{\infty}$ estimate for solutions in Euclidean balls.

\begin{lemma}\label{lem: subsol}
	Assume $0\leq g \in L^p(X)$ and $\|g\|_p\leq 1$. Then there exists $u\in \PSH(X,\omega)$ such that $-1\leq u\leq 0$, and $(\omega+dd^c u)^n \geq mgdV$ on $X$, where $m>0$ is a constant depending on $p$, $n$,  $\omega$. 
\end{lemma}
We stress that, as it will be clear from the proof, $m$ depends on the norm $\|g\|_p$, so if the norm is less than $1$ then the constant $m$ is uniform.

\begin{proof}
	 Consider a finite double cover of $X$ by small ``balls" $B_j,B_j'=\{\rho_j<0\}$, $j=1, \cdots, N$, with $B_j \subset \subset B_j'$
    which are bounded domains in a local holomorphic chart. Here $\rho_j: X \rightarrow \bR$ denotes a smooth function
  which is strictly plurisubharmonic in a neighborhood of $B_j'$.
    We solve $(dd^c v_j)^n=\mathbf{1}_{B_j'} gdV$ in $B_j' $ with $-1$ boundary values.
    It follows from Theorem \ref{thm: Linfty estimate}  that the plurisubharmonic solution $v_j$ is uniformly bounded in $B_j'$ by a constant which only depends on ${\it diam}(B_j')$ and $p$.
    
  We now consider $w_j=\max(v_j,\lambda_j \rho_j)$ and choose $\lambda_j >1$ so that we have a uniform bounded function $w_j$ with the following properties:
    \begin{itemize}
    \item $w_j$ coincides with $v_j$ in $B_j$ where it satisfies $ (dd^c w_j)^n = g dV$;
    \item $w_j$ is plurisubharmonic in $B_j'$ and uniformly bounded;
    \item $w_j$ coincides with $\lambda_j \rho_j$ in $X \setminus B_j'$ and in a neighborhood of $\partial B_j'$.
    \end{itemize}
    As  $w_j$ is smooth where it is not plurisubharmonic, the form 
    $dd^c w_j$ is bounded below by $-\delta^{-1} \omega$ for some uniform $\delta>0$. Note that such $\delta$ is uniform in $j$ since $X\setminus B_j'$ is compact.
    Thus $\delta w_j$ is $\omega$-psh. Consider the (uniformly) bounded $\omega$-psh function
    $
    v:=\frac{\delta}{N} \sum_{j=1}^N w_j
    $. In each $B_j$ we have
    $$
    (\omega+dd^c v)^n = \left(  \frac{1}{N}\sum_{l=1}^N (\omega+\delta dd^c w_l)\right)^n \geq \frac{1}{N^n} (\omega+dd^c \delta w_j)^n
    \geq \frac{\delta^n}{N^n} ( dd^c w_j)^n= \frac{\delta^n}{N^n} gdV
    $$
 where in the second inequality we used that $w_j$ is psh in $B_j$.
 
   The desired bounded subsolution can be given by $v/M$, where $M:= \max({\rm osc}_X v,1)$. 
\end{proof}

\begin{theorem}\label{Linfty est}
	Let $u\in \PSH(X,\omega)\cap L^{\infty}$ solve 
	\[
	(\omega+dd^c u)^n = f \omega^n,
	\]
	where $0\leq f\in L^p(X)$. Then $\int_X f \omega^n>0$ and ${\rm ocs}_X(u) \leq C$, where $C$ depends on $X,\omega, n,p$ and $\|f\|_p$. 
\end{theorem}
\begin{proof}
    That $\int_X f\omega^n>0$ follows from Corollary \ref{cor: dom princ}. W.l.o.g. we can assume that $\sup_X u=0$.  It follows from H\"older inequality and Skoda's uniform integrability theorem 
(see \cite[Theorem 8.11]{GZbook}) that one can find $\varepsilon>0$, $p'=p'(\varepsilon,p) \in (1,p)$, 
and $C=C({\varepsilon},p)>0$ independent of $u$ such that
$g=e^{-\varepsilon u} f \in L^{p'}$ with
$$
||g||_{p'} \leq ||f||_p \cdot ||e^{-\varepsilon \frac{p}{p-p'}u}||_{\frac{p}{p-p'}} \leq C({\varepsilon},p) \|f\|_p.
$$
Let $-1\leq v\leq 0$ be the bounded subsolution provided by Lemma \ref{lem: subsol} for the density
$\frac{g}{||g||_{p'}}$. Using that $v \leq 0$ we obtain
$$
(\omega+dd^c v)^n \geq m \frac{g}{||g||_{p'}} dV \geq m' e^{-\varepsilon u} f dV\geq  e^{\varepsilon(v-u + \varepsilon^{-1}\log m')}(\omega+dd^c u)^n,
$$
where $m'$ depends on $\|f\|_p$. In particular, setting $\tilde{v}:=v + \varepsilon^{-1}\log m'$, we find
$$  (\omega+dd^c u)^n \leq    e^{-\varepsilon(\tilde{v}-u)} (\omega+dd^c \tilde{v})^n.$$
 It then follows from Corollary \ref{cor:MAexp a} that 
\[
u \geq  v+\varepsilon^{-1}\log m' \geq -C:=-1+\varepsilon^{-1}\log m',
\]
since $v\geq -1$.
\end{proof}

\subsection{Laplacian estimates}
\begin{theorem}\label{thm: Lap est}
    Let $\varphi\in \PSH(X,\omega)\cap \mathcal{C}^{4}(X)$ solve     \[
    (\omega+dd^c \varphi)^n = e^{\lambda(\varphi-g)+f} \omega^n,
    \]
    where $\lambda \geq 0$ is a fixed constant, and $f,g$ are smooth functions.
    Then $\Delta_{\omega}\varphi \leq C$, where $C$ only depends on $X,\omega, \sup_X|\varphi|$, and $\|f\|_{C^2(X)}$, $\|g\|_{C^2(X)}$. 
\end{theorem}

We stress here that the constant $C$ does not depend on the parameter $\lambda \geq 0$. This will be crucial when obtaining smooth solutions for $\lambda=0$ and the $C^{1,\bar{1}}$-regularity of the envelope $P_{\omega}(g)$.

\begin{proof}
    We use $C$ to denote various uniform constants which may be different from one line to another.
 We follow the computations of \cite[proof of Theorem 2.1]{TW10a} and \cite{To18}. 
  Consider
 $$
 H:=\log {\rm Tr}_{\omega}(\tilde{\omega})- \gamma(u)
 $$
 where 
 $$
\tilde{\omega}:=\omega+dd^c \varphi, \; u:= \varphi-\inf_X \varphi +1,
 $$
 and $\gamma: \mathbb{R} \rightarrow \mathbb{R}$ is a smooth concave increasing function.
 % which will be chosen later. 
 We are going to show that $H$ is uniformly bounded from above for an appropriate choice of
 $\gamma$. Since $u$ is uniformly bounded, this will
 yield a uniform bound
 $$
 \Delta_{\omega}\varphi=
 {\rm Tr}_{\omega}(\omega+dd^c \varphi) - n\leq C.
 $$
 
 We let ${g}$ denote the Riemannian metric associated to $\omega$
 and $\tilde{g}$ the one asso\-ciated to $\tilde{\omega}=\omega+dd^c \varphi$. 
 The maximum of $H$ is attained at some point $x_0 \in X$ since $X$ is compact and $H$ is continuous. Since we want to bound $\Tr_{\omega}(\tilde{\omega})$ from above we can assume that it is greater than $1$.
 We use special 
 %local 
 coordinates at this point, as in Guan-Li \cite{GL10}: 
 \[
 g_{i\bar{j}} = \delta_{ij}, \; \;  
 \frac{\partial g_{i\bar{i}}}{\partial z_j}=0 
 \; \; \text{and}\; \; \tilde{g}_{i\bar{j}}\; \text{is diagonal}.
 \]
 To achieve this we use a linear change of coordinates so that $g_{i\bar{j}}=\delta_{ij}$ and $\tilde{g}_{i\bar{j}}$ is diagonal at $x_0$,
 and we then make a change of coordinates as in \cite[(2.19)]{GL10}. 
In the following, $g_{j\bar{k}i}, g_{j\bar{k}\bar{i}l}$ are ordinary partial derivatives.
 We first compute 
 \begin{flalign}
 	\Delta_{\tilde{\omega}} {\rm Tr}_{\omega}(\tilde{\omega})
 	&= \sum_{i,j,k,l} \tilde{g}^{i\bar{j}}\partial_i \partial_{\bar{j}} (g^{k\bar{l}} \tilde{g}_{k\bar{l}})\nonumber \\
 	&=\sum_{i,k}\tilde{g}^{i\bar{i}} \tilde{g}_{k\bar{k}i\bar{i}} -2 \Re\left (\sum_{i,j,k} \tilde{g}^{i\bar{i}}g_{j\bar{k}\bar{i}} \tilde{g}_{k\bar{j}i}\right) + \sum_{i,j,k} \tilde{g}^{i\bar{i}}g_{j\bar{k}i} g_{k\bar{j}\bar{i}} \tilde{g}_{k\bar{k}}\nonumber \\
 	&+\sum_{i,j,k} \tilde{g}^{i\bar{i}}g_{k\bar{j}i} g_{j\bar{k}\bar{i}} \tilde{g}_{k\bar{k}} -\sum_{i,k} \tilde{g}^{i\bar{i}}g_{k\bar{k}\bar{i}i}\tilde{g}_{k\bar{k}}\nonumber\\
 	&\geq \sum_{i,k}\tilde{g}^{i\bar{i}} \tilde{g}_{k\bar{k}i\bar{i}} -2 \Re\left (\sum_{i,j,k} \tilde{g}^{i\bar{i}}g_{j\bar{k}\bar{i}} \tilde{g}_{k\bar{j}i}\right) -C{\rm Tr}_{\omega}(\tilde{\omega}){\rm Tr}_{\tilde{\omega}}(\omega). \nonumber %\label{eq: C2 est 1} 
 \end{flalign}

Using the above, the fact that $|\tilde{g}_{i\bar{i}k\bar{k}}-\tilde{g}_{k\bar{k}i\bar{i}}|\leq C$ and 
\[
{\rm Tr}_{\omega} {\rm Ric}(\tilde{\omega})=\sum_{i,k} \tilde{g}^{i\bar{i}} \left (-\tilde{g}_{i\bar{i}k\bar{k}} + \sum_j \tilde{g}^{j\bar{j}}\tilde{g}_{i\bar{j}k}\tilde{g}_{j\bar{i}\bar{k}}\right)
\]
we obtain
\begin{flalign}
		\Delta_{\tilde{\omega}} {\rm Tr}_{\omega}(\tilde{\omega})& \geq \sum_{i,j,k} \tilde{g}^{i\bar{i}} \tilde{g}^{j\bar{j}}\tilde{g}_{i\bar{j}k}\tilde{g}_{j\bar{i}\bar{k}}-{\rm Tr}_{\omega} {\rm Ric}(\tilde{\omega}) \label{eq: C2 est 1} -{C \Tr_{\tilde{\omega}}(\omega)} \\
		&-C{\rm Tr}_{\omega}(\tilde{\omega}){\rm Tr}_{\tilde{\omega}}(\omega)  
		 -2 \Re\left (\sum_{i,j,k} \tilde{g}^{i\bar{i}}g_{j\bar{k}\bar{i}} \tilde{g}_{k\bar{j}i}\right)\nonumber.
\end{flalign}
 Here ${\rm Ric}(\tilde{\omega})$ is the Chern-Ricci form of $\tilde{\omega}$. 
%From ${\rm Tr}_{\omega}(\tilde{\omega}){\rm Tr}_{\tilde{\omega}}(\omega)\geq n$ \textcolor{red}{i do not understand where you use it in the next line, there is a constant $C_1>0$ that will be absorbed inside ${\rm Tr}_{\omega}(\tilde{\omega}){\rm Tr}_{\tilde{\omega}}(\omega)$ but it is not important.} 
 From \eqref{eq: C2 est 1} and 
$$
{\rm Ric}(\tilde{\omega})= {\rm Ric}(\omega) - \lambda  dd^c (\varphi-g) - dd^c f \leq   C \omega - \lambda\tilde{\omega} +\lambda C\omega \leq C(1+\lambda) \omega - \lambda \tilde{\omega},
$$
we obtain 
\begin{flalign}
	\Delta_{\tilde{\omega}} {\rm Tr}_{\omega}(\tilde{\omega})& \geq \sum_{i,j,k} \tilde{g}^{i\bar{i}} \tilde{g}^{j\bar{j}}\tilde{g}_{i\bar{j}k}\tilde{g}_{j\bar{i}\bar{k}} -C(1+\lambda)n + \lambda \Tr_{\omega}(\tilde{\omega}) -C \Tr_{\tilde{\omega}}(\omega)\label{eq: C2 est 2}\\
		&  -2 \Re\left (\sum_{i,j,k} \tilde{g}^{i\bar{i}}g_{j\bar{k}\bar{i}} \tilde{g}_{k\bar{j}i}\right)  - C{\rm Tr}_{\omega}(\tilde{\omega}){\rm Tr}_{\tilde{\omega}}(\omega). \nonumber
        \end{flalign}
Our special choice of coordinates at $x_0$ ensures that $g_{j\bar{j}\bar{i}}=0$. 
Using Cauchy-Schwarz inequality and $|\tilde{g}_{k\bar{j}i}-\tilde{g}_{i\bar{j}k}|\leq C$, we therefore obtain

\begin{flalign*}
	\left | 2 \Re\left ( \sum_{i,j,k} \tilde{g}^{i\bar{i}}g_{j\bar{k}\bar{i}} \tilde{g}_{k\bar{j}i}\right)\right| 
	&\leq  \left | 2 \Re\left (\sum_{i}\sum_{j\neq k} \tilde{g}^{i\bar{i}}g_{j\bar{k}\bar{i}} \tilde{g}_{i\bar{j}k}\right) \right | 
	+ C \Tr_{\tilde{\omega}}(\omega)\\
	&\leq \sum_{i}\sum_{j\neq k}  \left ( \tilde{g}^{i\bar{i}}\tilde{g}^{j\bar{j}} \tilde{g}_{i\bar{j} k} \tilde{g}_{j\bar{i}\bar{k}}+  \tilde{g}^{i\bar{i}}\tilde{g}_{j\bar{j}} g_{j\bar{k}\bar{i}} g_{k\bar{j}i}\right) + C \Tr_{\tilde{\omega}}(\omega)\\
	& \leq \sum_{i}\sum_{j\neq k} \tilde{g}^{i\bar{i}}\tilde{g}^{j\bar{j}} \tilde{g}_{i\bar{j}k} \tilde{g}_{j\bar{i}\bar{k}} + C {\rm Tr}_{\omega}(\tilde{\omega}){\rm Tr}_{\tilde{\omega}}(\omega)+ C \Tr_{\tilde{\omega}}(\omega).
\end{flalign*} 
Together with \eqref{eq: C2 est 2} this yields
% \begin{flalign}
\begin{equation}\label{eq: C2 estimate 3_a}
 \Delta_{\tilde{\omega}}
 {\rm Tr}_{\omega} (\tilde{\omega}) 
\geq   I  -C {\rm Tr}_{\omega} (\tilde{\omega})  {\rm Tr}_{\tilde{\omega}}  (\omega) -C(1+\lambda)n + \lambda \Tr_{\omega}(\tilde{\omega})-C  \Tr_{\tilde{\omega}}(\omega),
\end{equation}
% \end{flalign}
 with $I:=\sum_{i,j} \tilde{g}^{i \bar{i}}  \tilde{g}^{j \bar{j}} \tilde{g}_{i \bar{j}j}   \tilde{g}_{j \bar{i} \bar{j}}$. We next compute 
\begin{flalign*}
	|\partial \Tr_{\omega}(\tilde{\omega}) |_{\tilde{\omega}}^2&=\sum_{i,j,k} \tilde{g}^{i\bar{i}} \tilde{g}_{j\bar{j}i} \tilde{g}_{k\bar{k}\bar{i}}\\
	&= \sum_{i,j,k} \tilde{g}^{i\bar{i}}(T_{ij\bar{j}}+ \tilde{g}_{i\bar{j}j}) (\overline{T}_{ik\bar{k}}+\tilde{g}_{k\bar{i}\bar{k}})\\
	& =\sum_{i,j,k} \tilde{g}^{i\bar{i}} \tilde{g}_{i\bar{j}j} \tilde{g}_{k\bar{i}\bar{k}}  + \sum_{i,j,k} \tilde{g}^{i\bar{i}} T_{ij\bar{j}} \overline{T}_{ik\bar{k}} + 2\Re \left ( \sum_{i,j,k}\tilde{g}^{i\bar{i}} T_{ij\bar{j}} \tilde{g}_{k\bar{i} \bar{k}} \right),
\end{flalign*}
 where $T_{ij\bar{j}}= \tilde{g}_{j\bar{j}i}-\tilde{g}_{i\bar{j}j}=\partial_j g_{\bar{j}i} -\partial_i g_{\bar{j}j}$ is the torsion term of $\tilde{\omega}$ which is under control: $|T_{ij\bar{j}}|\leq C$. We bound the first term by Cauchy-Schwarz inequality 
 \begin{flalign*}
 	\sum_{i,j,k} \tilde{g}^{i\bar{i}} \tilde{g}_{i\bar{j}j} \tilde{g}_{k\bar{i}\bar{k}} &= \sum_{i} \tilde{g}^{i\bar{i}} \left |\sum_{j} \tilde{g}_{i\bar{j}j}\right|^2 \\
 	&\leq \left (\sum_{i,j} \tilde{g}^{i\bar{i}} \tilde{g}^{j\bar{j}} \tilde{g}_{i\bar{j}j} \tilde{g}_{j\bar{i}\bar{j}} \right) \left (\sum_{j} \tilde{g}_{j\bar{j}}  \right) =I\Tr_{\omega}(\tilde{\omega}). 
 \end{flalign*}
 We thus get
  \begin{equation}\label{eq: C2 estimate 4}
\frac{| \partial {\rm Tr}_{\omega} (\tilde{\omega}) |^2_{\tilde{\omega}}}{({\rm Tr}_{\omega} (\tilde{\omega}))^2}
\leq \frac{I}{{\rm Tr}_{\omega} (\tilde{\omega})}
+C\frac{ {\rm Tr}_{\tilde{\omega}}(\omega)}{({\rm Tr}_{\omega} (\tilde{\omega}))^2}
+\frac{2}{({\rm Tr}_{\omega} (\tilde{\omega}))^2}
\Re\left(\sum_{i,j,k} \tilde{g}^{i \overline{i}}  T_{ij\overline{j}}  \tilde{g}_{k\bar{i}\bar{k}}  \right).
 \end{equation}
 Since $\partial_{\overline{i}} H=0$ at the point $x_0$, we obtain by differentiating $H$ once
  $$
\sum_k \tilde{g}_{k\overline{k}\overline{i}}=   {\rm Tr}_{\omega} (\tilde{\omega})\gamma'(u)u_{\overline{i}}. 
 $$
The above together with  $|T_{ij\bar{j}}|\leq C$ and Cauchy-Schwarz inequality yield
  \begin{equation*}
\left | \frac{2}{({\rm Tr}_{\omega} (\tilde{\omega}))^2}
\Re\left(\sum \tilde{g}^{i \overline{i}}  T_{ij\overline{j}}  \tilde{g}_{k \overline{k} \overline{i}}  \right) \right |
\leq C \frac{\gamma'(u)^2}{(-\gamma''(u))}
\frac{{\rm Tr}_{\tilde{\omega}}(\omega)}{({\rm Tr}_{\omega} (\tilde{\omega}))^2}
+ (-\gamma''(u)) |\partial u|_{\tilde{\omega}}^2.
 \end{equation*}
Using once again that $|\tilde{g}_{k\bar{k}\bar{i}}-\tilde{g}_{k\bar{i}\bar{k}}|\leq C$ we infer
\begin{eqnarray}\label{eq:CS trace}
\left | \frac{2}{({\rm Tr}_{\omega} (\tilde{\omega}))^2}
\Re\left(\sum \tilde{g}^{i \overline{i}}  T_{ij\overline{j}}  \tilde{g}_{k \overline{i} \overline{k}}  \right) \right |
&\leq & C \left (\frac{\gamma'(u)^2}{(-\gamma''(u))} +1\right)
\frac{{\rm Tr}_{\tilde{\omega}}(\omega)}{({\rm Tr}_{\omega} (\tilde{\omega}))^2}\\
\nonumber &&+ (-\gamma''(u)) |\partial u|_{\tilde{\omega}}^2.
\end{eqnarray}
Since $0 \geq \Delta_{\tilde{\omega}} H$  at $x_0$, it follows from 
 \eqref{eq: C2 estimate 3_a} and \eqref{eq: C2 estimate 4} that
  \begin{flalign}\label{eq: C2 est final 1_a}
0 &\geq \Delta_{\tilde{\omega}} H
= 
 \frac{\Delta_{\tilde{\omega}} {\rm Tr}_{\omega} (\tilde{\omega})}{ {\rm Tr}_{\omega} (\tilde{\omega})}
 -\frac{| \partial {\rm Tr}_{\omega} (\tilde{\omega}) |^2_{\tilde{\omega}}}{({\rm Tr}_{\omega} (\tilde{\omega}))^2}
 -\gamma'(u) \Delta_{\tilde{\omega}}(u) - \gamma''(u) |\partial u|_{\tilde{\omega}}^2  \\
 &\geq   
-C \left (\frac{\gamma'(u)^2}{(-\gamma''(u))}+1\right)
\frac{{\rm Tr}_{\tilde{\omega}}(\omega)}{({\rm Tr}_{\omega} (\tilde{\omega}))^2} + \lambda - \frac{C\lambda}{\Tr_{\omega}(\tilde{\omega})}\nonumber \\
&+(\gamma'(u)-C)  \Tr_{\tilde{\omega}}(\omega)-C \frac{\Tr_{\tilde{\omega}}(\omega)}{\Tr_{\omega}(\tilde{\omega})}.\nonumber
 \end{flalign}
 Recall that we have assumed $\Tr_{\omega}(\tilde{\omega})\geq 1$. 
We now choose the function $\gamma$ so as to obtain a simplified information. We set
 $$
 \gamma(u):= (C+1)u + \log u. 
 $$
 Since $u$ is uniformly bounded and $u\geq 1$, we observe that
 $$
 C+1 \leq \gamma'(u) \leq  C+2
\; \; \; \text{ and }\;  \; \;
  \frac{\gamma'(u)^2}{|\gamma''(u)|} +1 \leq C_1.
 $$
By incorporating  this into \eqref{eq: C2 est final 1_a}  we obtain
\begin{flalign*}
	0 &\geq  -C_2+ (C+1)\Tr_{\tilde{\omega}}(\omega)   \\
	&-C_2 
\frac{{\rm Tr}_{\tilde{\omega}}(\omega)}{({\rm Tr}_{\omega} (\tilde{\omega}))^2}- \frac{C\lambda}{\Tr_{\omega}(\tilde{\omega})} + \lambda-C \Tr_{\tilde{\omega}}(\omega)-C \frac{\Tr_{\tilde{\omega}}(\omega)}{\Tr_{\omega}(\tilde{\omega})}.
\end{flalign*}
We thus arrive at 
 \begin{equation}
 	\label{eq: C2 est final 2}
 	 0 \geq \Tr_{\tilde{\omega}}(\omega) -C_2\frac{{\rm Tr}_{\tilde{\omega}}(\omega)}{({\rm Tr}_{\omega} (\tilde{\omega}))^2}+\lambda - \frac{C\lambda}{\Tr_{\omega}(\tilde{\omega})}- C \frac{\Tr_{\tilde{\omega}}(\omega)}{\Tr_{\omega}(\tilde{\omega})} -C_2. 
 \end{equation}
   At the point $x_0$ we have the following alternative. 
\medskip

\noindent \textbf{Case 1.}
If ${\rm Tr}_{\omega}(\tilde{\omega})^2 \geq 4C_2 +(4C)^2$ then 
$$
C_2\frac{{\rm Tr}_{\tilde{\omega}}(\omega)}{({\rm Tr}_{\omega} (\tilde{\omega}))^2} \leq \frac{{\rm Tr}_{\tilde{\omega}}(\omega)}{4}
\; \; \text{and} \; \;
 C\frac{\Tr_{\tilde{\omega}}(\omega)}{\Tr_{\omega}(\tilde{\omega})} \leq \frac{{\rm Tr}_{\tilde{\omega}}(\omega)}{4},
$$
hence from \eqref{eq: C2 est final 2} we get 
\begin{equation}\label{eq: C2 final}
    0 \geq \frac{\Tr_{\tilde{\omega}}(\omega)}{4}  + (\lambda -C_2) - \frac{C\lambda}{\Tr_{\omega}(\tilde{\omega})} \geq \frac{\Tr_{\tilde{\omega}}(\omega)}{4}  -C_2.   
\end{equation}
 Now, if $\lambda\leq 2C_2$, then
\begin{equation*}%\label{eq: compare two trace}
\Tr_{\omega}(\tilde{\omega}) 
\leq n \frac{\tilde{\omega}^n}{\omega^n} \left( \Tr_{\tilde{\omega}}(\omega) \right)^{n-1}
\leq  n(4C_2)^{n-1}  e^{\lambda(\varphi-g)+f}
\end{equation*}
yields ${\rm Tr}_{\omega}(\tilde{\omega}) \leq C$. If $\lambda\geq 2C_2$, then the first inequality in \eqref{eq: C2 final} yields 
\[
0 \geq \frac{\Tr_{\tilde{\omega}}(\omega)}{4}  + \frac{\lambda}{2}  - \frac{C\lambda}{\Tr_{\omega}(\tilde{\omega})} \geq \frac{\lambda}{2}  - \frac{C\lambda}{\Tr_{\omega}(\tilde{\omega})}.
\]
It follows that $\Tr_{\omega}(\tilde{\omega}) \leq 2C$, hence
\begin{flalign*}
H(x_0) &\leq \log (2C)  -(C+1) u\leq  	C_3. 
\end{flalign*}
%where in the last line we have used 
%by using \eqref{eq: C0 estimate}. 
\smallskip

\noindent {\bf Case 2.}
If ${\rm Tr}_{\omega}(\tilde{\omega})^2 \leq 4C_2 +(4C)^2$ then
$$
 H(x_0) \leq \log \sqrt{4C_2 +(4C)^2}- \gamma(u) \leq C_4.
 $$

 \medskip

Thus $H(x_0)$ is uniformly bounded from above, yielding the desired estimate. 
\end{proof}

\subsection{Higher order estimates}\label{sec:higher est}
Higher-order estimates can be derived using standard arguments of elliptic and complex Monge-Amp\`ere theories \cite{Aub78, LU}. We also mention \cite{TW10a} for more details in the hermitian setting. 

\section{Solving complex Monge-Amp\`ere equations}
The goal of this section is to solve complex Monge-Amp\`ere equations on compact Hermitian manifolds. We will use the continuity method to solve the following equation 
\begin{equation}\label{eq:MA}
    (\omega+dd^c \varphi)^n = e^{\lambda \varphi +f} \omega^n,
\end{equation}
where $f$ is smooth and $\lambda \geq 0$ is a constant. 

\subsection{Uniqueness of smooth solutions}
For $\lambda>0$ the uniqueness is a direct consequence of the maximum principle. Assume now that $\lambda=0$ and $\varphi, \psi$ are smooth solutions of \eqref{eq:MA}. From the Monge-Amp\`ere equations for $\varphi$ and $\psi$, we infer that
\[
dd^c (\varphi-\psi) \wedge T = 0,  \text{where}\; T= \sum_{k=0}^{n-1} \omega_{\varphi}^k \wedge \omega_{\psi}^{n-1-k}. 
\]
From the strong maximum principle, see \cite[Theorem 2.10]{HanLin11}, it follows that $\varphi-\psi$ is constant.

\subsection{Existence for $\lambda>0$}\label{existence pos case}
To simplify the notation, we assume that $\lambda=1$. Consider the following family of equations 
\begin{equation}\tag{MA$_t$}\label{eq: MAt}
    (\omega+dd^c \varphi_t)^n = e^{\varphi_t + t f } \omega^n,
\end{equation}
and let $I$ be the set of all $t\in [0,1]$ such that the equation \eqref{eq: MAt} admits a solution $\varphi_t \in \mathcal{C}^{k,\alpha}$, where $k\geq 4$ is a fixed integer, and $\alpha\in (0,1)$. It is clear that $0\in I$ as $\varphi_0=0$ is a solution. We will prove that $I$ is both open and closed, and since $[0,1]$ is connected, we infer that $I=[0,1]$, which implies that the equation \eqref{eq:MA} has a solution in $\mathcal{C}^{k,\alpha}$. By uniqueness, we infer that the solution is in fact smooth (in $\mathcal{C}^{k,\alpha}$ for all $k,\alpha$).  

\subsubsection{$I$ is open}
Consider the map between Banach spaces $\mathcal{M}  : \mathcal{C}^{k,\alpha}(X) \rightarrow \mathcal{C}^{k-2,\alpha}(X)$ defined by 
\[
\mathcal{M}(\varphi) = \log \left ( \frac{(\omega+dd^c \varphi)^n}{\omega^n}  \right)- \varphi. 
\]
The derivative of $\mathcal{M}$ at $\varphi$ is given by 
\[
d\mathcal{M}_{\varphi}(\chi) =\frac{d}{dt}\big|_{t=0} \mathcal{M}(\varphi+t\chi) = \Delta_{\omega_{\varphi}}\chi - \chi, \qquad \chi \in C^\infty (X, \R).
\]
Given $f\in \mathcal{C}^{k-2,\alpha}(X)$, the equation 
\[
\Delta_{\omega_{\varphi}} \chi -\chi= f 
\]
can be written as 
\begin{equation}\label{eq lap}
\Delta_{\omega'} \chi - e^{-g} \chi = e^{-g}f,
\end{equation}
where $\omega' = e^{g} (\omega+dd^c \varphi)$ is the unique  Gauduchon metric in the conformal class of $\omega+dd^c \varphi$.  Since $e^{-g}>0$, a distributional solution always exists and is in $C^{k,\alpha}$, as follows from the classical Schauder theory of linear elliptic equation \cite{GT01}. 
Let us observe that since reference form $\omega'$ is Gauduchon, (using integration by parts and the fact that $dd^c (\omega')^{n-1}=0$) one can prove that solving \eqref{eq lap} is equivalent to maximize the functional $E: \big\{u \,:\, \int_X u \, (\omega')^n=0 \big\}\rightarrow \R$ defined as  $E(u):=\int_X |\nabla u|^2 (\omega')^n$.

It then follows that $d\mathcal{M}_{\varphi}$ is an isomorphism between the two Banach spaces $\mathcal{C}^{k,\alpha}(X)$ and $\mathcal{C}^{k-2,\alpha}(X)$. Thus, the map $\mathcal{M}$ is locally invertible at any $\varphi\in \mathcal{C}^{k,\alpha}(X)$ which is strictly $\omega$-psh. This means that, for given $t\in I$ and $s$ sufficiently near to $t$, the equation $\mathcal{M}(\varphi)=s f$ admits a solution in $\mathcal{C}^{k,\alpha}$, establishing the openness of $I$. 

\subsubsection{$I$ is closed}
Let $t_i\in I$, i.e. assume $\varphi_{t_i}$ is a $C^{k, \alpha}$ solution of \eqref{eq: MAt} at time $t=t_i$. Let $t_i\rightarrow t_*$, $t_*\in [0,1]$. We want to show that $t_*\in I$. In order to that we need to establish uniform estimates for $\|\varphi_{t_i}\|_{C^{k,\alpha}}$. Ascoli-Arzela's theorem will then ensure that $\varphi_{t_i}$ converges to $\varphi_{t_*}$ in $C^{k,\alpha}$, which will be solution of \eqref{eq: MAt} at time $t=t_*$, hence $t_*\in I$.
The $L^\infty$-estimate is a direct consequence of the classical maximum principle. Let $x_{0,i}$ be the maximum point of $\varphi_{t_i}$. Then $dd^c \varphi_{t_i}(x_{0,i})\leq 0$, hence $ e^{\varphi_{t_i}+t_i f} \omega^n=\omega_{\varphi_{t_i}}^n \leq \omega^n$ at $x_{0,i}$ meaning that $(\varphi_{t_i}+t_i f)(x_{0,i})\leq 0$, hence the uniform upper bound for $\varphi_{t_i}$.
The uniform laplacian estimate will then follow from Theorem \ref{thm: Lap est} and the higher order estimates from Section \ref{sec:higher est}.

\subsection{Quasi-plurisubharmonic envelopes}

In this section we prove a crucial orthogonal property of quasi-psh envelopes. 
\begin{theorem}
    Let $g$ be a smooth function on $X$, $\lambda>1$, and let $\varphi_{\lambda}$ solve the Monge-Ampère equation
     \[
    (\omega+dd^c \varphi)^n = e^{\lambda (\varphi-g)}\omega^n.
    \]
Then $\varphi_{\lambda}$ uniformly converge to the envelope $P_{\omega}(g)$ as $\lambda\to +\infty$. Consequently, $\Delta_{\omega}P_{\omega}(g)\leq C$ is bounded and $(\omega+dd^c P_{\omega}(g))^n$ is supported on the contact set $\{P_{\omega}(g)=g\}$. 
\end{theorem}
   The above result is due to Berman \cite{Ber19} (see also Chu-Zhou \cite{CZ19} for an adaptation of Berman's proof in the Hermitian case).
   \begin{proof}
First of all, observe that by Section \ref{existence pos case}, $\varphi_\lambda$ is smooth. Let $x_0\in X$ be a point where $\varphi_{\lambda}-g$ attains its maximum over $X$. By the classical maximum principle, we have $e^{\lambda(\varphi_{\lambda}-g)} \omega^n \leq (\omega+dd^c g)^n$ at $x_0$. It then follows that $\varphi_{\lambda} -g \leq  \varphi_{\lambda} -g (x_0)\leq \lambda^{-1} \log \frac{\omega_g^n}{\omega^n}|_{x_0}:= C\lambda^{-1}$, for a uniform constant $C$, which yields $\varphi_{\lambda} \leq P_{\omega}(g)+ C\lambda^{-1}$. Now, the function 
       \[
       u= (1-\lambda^{-1})P_{\omega}(g) + \lambda^{-1} \inf_X g  - n \lambda^{-1}\ln \lambda
       \]
     satisfies 
       \begin{equation}\label{ineq_env}
       (\omega+dd^c u)^n \geq e^{\lambda(u-g)}\omega^n. 
       \end{equation}
       Indeed, 
       $$ (\omega+dd^c u)^n = (\omega+dd^c(1-\lambda^{-1})P_{\omega}(g))^n \geq \lambda^{-n} \omega^n$$
       and one can check that $\lambda^{-n} \geq e^{\lambda (u-g)}$.
       In particular, \eqref{ineq_env} implies that for any $\varepsilon>0$, $${\bf 1}_{\{\varphi_\lambda <u-\varepsilon\}} \omega_{\varphi_\lambda}^n = {\bf 1}_{\{\varphi_\lambda <u-\varepsilon\}} e^{\lambda (\varphi-g)}\omega^n  \leq {\bf 1}_{\{\varphi_\lambda <u-\varepsilon\}}  e^{\lambda (u-\varepsilon-g)}\omega^n \leq e^{-\lambda \varepsilon}  {\bf 1}_{\{\varphi_\lambda <u-\varepsilon\}}\omega_u^n $$
    The domination principle Theorem \ref{thm: domination principle} (applied with $c=e^{-\lambda \varepsilon}<1$) then yields $\varphi_{\lambda}\geq u-\varepsilon$. Sending $\varepsilon$ to zero we get $\varphi_{\lambda}\geq u$.
    
    This allows to conclude that $\varphi_{\lambda}$ uniformly converge to $P_{\omega}(g)$ as $\lambda \to +\infty$. In particular \cite[Theorem 4.26]{GZbook} ensures that the Monge-Amp\`ere measure $(\omega+dd^c \varphi_\lambda)^n$ weakly converges to $(\omega+dd^c P_{\omega}(g))^n$. The Laplacian estimate (Theorem \ref{thm: Lap est}) ensures that $\Delta_{\omega}\varphi_{\lambda}$ is  bounded by a uniform constant, hence so is $\Delta_{\omega}P_{\omega}(g)$. For a fixed $\varepsilon>0$ (and $\lambda$ big enough), we also have 
    \[
    (\omega+dd^c \varphi_{\lambda})^n (P_{\omega}(g) < g -2\varepsilon) \leq (\omega+dd^c \varphi_{\lambda})^n (\varphi_{\lambda} < g -\varepsilon) \leq \int_X e^{-\varepsilon \lambda}\omega^n. 
    \]
    Letting $\lambda\to +\infty$ and then $\varepsilon\to 0$, we arrive at $(\omega+dd^c P_{\omega}(g))^n(P_{\omega}(g)<g)=0$, finishing the proof. 
   \end{proof}

\begin{prop}
    If $f$ is quasi lower-semicontinuous on $X$, then $(\omega+dd^c P_{\omega}(f))^n$ is supported on the contact set $\{P_{\omega}(f)=f\}$.
\end{prop}
\begin{proof}
    By approximation and by using the above theorem. 
\end{proof}
%\begin{prop}
  %  Monge-Ampère of $P_{\omega}(u,v)$ 
%\end{prop}

\subsection{The case $\lambda=0$}
\begin{theorem}\label{thm: lambda 0}
    Given a smooth function $f$, there exists a unique constant $b$ and a unique smooth solution $\varphi \in \PSH(X,\omega)$ of the equation 
    \[
    (\omega+dd^c \varphi)^n =e^{f+b} \omega^n, \; \sup_X \varphi=0.  
    \]
\end{theorem}
\begin{proof}
    For each $j$ let $\varphi_j$ solve
    \begin{equation}\label{eq appr}
    (\omega+dd^c \varphi_j)^n= e^{j^{-1}\varphi_j+f}\omega^n. 
    \end{equation}
    Let $\omega_g$ be the unique Gauduchon metric in the conformal class of $\omega$ and write $\omega_g^n=e^h \omega^n$, for some smooth function $h$. The AM-GM inequality gives the mixed Monge-Amp\`ere inequality 
    \[
    (\omega+dd^c \varphi_j) \wedge \omega_g^{n-1} \geq \exp \left (\frac{\varphi_j}{nj} + \frac{f}{n} + \frac{(n-1)h}{n} \right) \omega^n.  
    \]
   Setting $\psi_j= \varphi_j -\sup_X \varphi_j$, and integrating the above over $X$, we obtain 
   \begin{equation}\label{unif bdd sup}
   \exp \left( \frac{\sup_X\varphi_j}{nj}  \right ) \leq C.
   \end{equation}
   Here, we have used the fact that $\int_X  (\omega+dd^c \varphi_j) \wedge \omega_g^{n-1} = \int_X  \omega \wedge \omega_g^{n-1}$ (since $\omega_g$ is Gauduchon) and that the functions $\psi_j/j$ converge in $L^1$ to $0$. It thus follows that $\varphi_j/j$ is uniformly bounded from above. We rewrite the equation \eqref{eq appr} as 
   \[
   (\omega+dd^c \psi_j)^n =e^{j^{-1}\psi_j + b_j+ f}\omega^n = e^{\psi_j+f_j}\omega^n,
   \]
   where $f_j = -\psi_j + f + j^{-j}\psi_j + b_j$, and $b_j=j^{-1}\sup_X \varphi_j$. 
   Now, $b_j\leq C$ is uniformly bounded above by \eqref{unif bdd sup} and by the $L^{\infty}$ estimate of Theorem \ref{Linfty est},  $\psi_j$ is uniformly bounded since the $L^p$ norm of the right-hand side, which is $e^{j^{-1}\psi_j+b_j+f}$, is uniformly bounded. Also, by Lemma \ref{lem: subsol}, there exists $m>0$ and a bounded function $\eta\in \PSH(X,\omega)$ such that $(\omega+dd^c \eta)^n \geq m e^{f}\omega^n$. We stress that $m$ depends on $\|e^f\|_p$, $p>1$. We then find $\omega_{\psi_j}^n \leq e^{b_j-\log m} \omega_{\eta}^n$, and Corollary \ref{cor:MAconstant a} gives $b_j-\log m\geq 0$, hence $b_j$ is uniformly bounded. 
   
  Extracting a subsequence, we can assume that $\psi_j$ converges in $L^1$ and almost everywhere to $\varphi\in \PSH(X,\omega)$ and $b_j\to b\in \mathbb{R}$. Define 
   \[
   u_j := \left (\sup_{k\geq j} \psi_k\right )^*, \; v_j := P_{\omega}\left(\inf_{k\geq j} \psi_k\right).
   \]
   Recall that
   \[
   (\omega+dd^c \psi_j)^n = e^{\psi_j+f_j}\omega^n,
   \]
   where $f_j = -\psi_j + f + j^{-j}\psi_j + b_j$, and $b_j=j^{-1}\sup_X \varphi_j$. 
   Then $(u_j), (v_j)$ are sequences of $\omega$-psh functions such that $v_j\leq \psi_j\leq u_j$ and $u_j \searrow u$, $v_j\nearrow v$. In particular $u\geq v$ and using \cite[Lemma 2.8 and 2.9]{DDL4}, one finds,
   
   \[
   (\omega+dd^c u)^n \geq e^{u -\varphi + f+b}\omega^n, \; (\omega+dd^c v)^n \leq e^{v-\varphi + f+b}\omega^n.
   \]

   For any $\varepsilon>0$, we then find 
   $$\idd_{\{v< u-\varepsilon\}} \omega_v^n \leq \idd_{\{v< u-\varepsilon\}} e^{v+f+b}\omega^n \leq  \idd_{\{v< u-\varepsilon\}} e^{u-\varepsilon+f+b}\omega^n=  e^{-\varepsilon}\idd_{\{v< u-\varepsilon\}} \omega_u^n$$
   The domination principle (Theorem \ref{thm: domination principle}) yields $v\geq u-\varepsilon$ for all $\varepsilon>0$, hence $v\geq u$. Since $v\leq \varphi \leq u$ we arrive at $u=v=\varphi$ and $ (\omega+dd^c \varphi)^n = e^{\varphi + f+b}\omega^n $.
   
   Moreover, the a priori estimates ensure that $\psi_j$ is uniformly bounded in $C^{k,\alpha}$-norm for all $k,\alpha$; thus  $\varphi$ is smooth. 
\end{proof}
\begin{theorem}\label{MA zero case}
    Assume $0\leq f\in L^p(X)$ for some $p>1$ and $\int_X f \omega^n>0$. Then there exist $\varphi \in \PSH(X,\omega)\cap L^{\infty}$ and a unique constant $c>0$ solving 
    \[
    (\omega +dd^c \varphi)^n =cf \omega^n. 
    \]
\end{theorem}

In the K\"ahler case, integrating both sides of the above equation reveals
\[
c\int_X f \omega^n = \int_X (\omega+dd^c \varphi)^n = \int_X \omega^n,
\]
which means that the constant $c$ can be computed explicitly from $f$ and $\omega$. In the non-K\"ahler case, $c$ is also an unknown to be solved. 

\begin{proof}
We take convolution with smoothing kernels and we define $f_j:= f\star \rho_{1/j} +j^{-1}>0$ and we observe that $\|f_j\|_p$ is controlled by $\|f\|_p$. For each $j$, we consider $\varphi_j$, solution of 
$$(\omega+dd^c \varphi_j)^n= e^{j^{-1} \varphi_j + \log f_j} \omega^n.$$
By the exact the same arguments in the proof above we find that $\psi_j:=\varphi_j-\sup_X \varphi_j$ and $b_j=  j^{-1}{\sup_X \varphi_j}$ are uniformly bounded and $\psi_j$ converges in $L^1$ to $\varphi$ which solves
$$(\omega+dd^c \varphi)^n= e^{b} f \omega^n.$$
\end{proof}

\section{Degenerate reference forms}

Fix a semipositive smooth $(1,1)$-form $\theta$ on $X$ which is big, i.e. there exists $\rho \in \PSH(X,\theta)$ with analytic singularities such that $\theta +dd^c \rho \geq \delta \omega$, for some $\delta>0$. Recall that a qpsh function $\rho$ is said to have analytic singularities if there exists a constant $c>0$ such that locally on X,
$$\rho =c \log \sum_{j=1}^N |f_j|^2+g$$
where $g$ is bounded and $f_1,...,f_N$ are local holomorphic functions. Let $D$ be the analytic set such that $\rho$ is smooth in $\Omega:=X\setminus D$. We also normalize $\rho$ so that $\sup_X\rho=-1$.
\begin{lemma}\label{lem: subsol big}
	Assume $0\leq g \in L^p(X)$ and $\|g\|_p\leq 1$.  Then, there exists $u\in \PSH(X,\theta)$ such that $0\leq u\leq 1$, and $(\theta+dd^c u)^n \geq mgdV$ on $X$, where $m>0$ is a constant depending on $p$, $n$,  $\omega$.  
\end{lemma}

\begin{proof}
    Since $\rho$ is qpsh, it belongs to $L^r$ for any $r>1$. H\"older's inequality ensures
	that  $|\rho|^{2n} g \in L^q(\omega^n)$ for some $q>1$. 
	It follows from  Lemma \ref{lem: subsol} that there
	 exist a uniform $c_1>0$ and $v\in \PSH(X,\omega)\cap L^{\infty}(X)$ such that 
	$ \sup_X v =-1$ and 
	\[
	(\omega+dd^c v)^n \geq  c_1|\rho|^{2n} g\omega^n.
	\]
    Note that $c_1$ depends on $\| |\rho|\|_r$. Also the function $\delta v+\rho$ is $\theta$-psh since 
	$\theta+dd^c (\delta v+\rho) \geq \delta (\omega+dd^c v)$.
	Set $u:= -(\delta v+\rho)^{-1}=\chi \circ (\delta v+\rho)$ with $\chi(t)=-t^{-1}$ convex increasing  function from  $(-\infty, -1]\rightarrow [0,1]$. 
	Our normalizations ensure  $0 \leq u \leq 1$, and 
	 a direct computation   yields
	\[
	\theta+dd^c u \geq 
	\delta \chi' \circ (\delta v+\rho)   (\omega+dd^c v)  = \frac{\delta}{(\delta v+\rho)^{2}} (\omega +dd^c v).
	\]
	We infer $\omega_u^n \geq |\delta v+\rho|^{-2n} \delta^n c_1 |\rho|^{2n} g\omega^n$.
	 Since $v \leq -1$ is bounded  and $\rho\leq -1$, it follows that $\theta_u^n \geq m g\omega^n$ 
	 for some uniform constant $m>0$, finishing the proof.
\end{proof}

\begin{theorem}[Domination principle]\label{thm: domination principle big}
    Assume $u,v \in \PSH(X,\theta)\cap L^{\infty}$ satisfy
    \[
    \idd_{\{u<v\}}(\theta +dd^c u)^n\leq c \idd_{\{u<v\}} (\theta +dd^c v)^n,
    \]
    for some constant $0\leq c<1$. Then $u\geq v$. 
\end{theorem}

\begin{proof}
    Since the proof is very similar to that of Theorem \ref{thm: domination principle}, we only highlight the main steps. First, by replacing $v$ with $(1-\varepsilon)v+\varepsilon \rho$, we can assume that $u-v$ goes to $+\infty$ on $\partial \Omega$. Secondly, we write  
    \[
    (\theta +dd^c u)^n =(\theta_{\rho} + dd^c (u-\rho))^n, \quad \text{in}\; \Omega. 
    \]
    Now, the arguments of Theorem \ref{thm: domination principle} can be applied for $u-\rho$ and $v-\rho$, $\theta_{\rho}$ (instead of $u,v, \omega$). We leave the details to the interested reader. 
\end{proof}

We record the following direct consequences of the domination principle. 

\begin{corollary}\label{cor:MAconstant}
    Assume $u,v$ are bounded $\theta$-psh functions on $X$ such that 
    $$
    (\theta+dd^c u)^n \leq c (\theta+dd^c v)^n, 
    $$
    for some positive constant $c$. Then $c\geq 1$. 
\end{corollary}

\begin{corollary}\label{cor:MAexp}
    Assume $u,v$ are bounded $\theta$-psh functions on $X$ such that 
    $$
    (\theta+dd^c u)^n \leq e^{\lambda(u-v)} (\theta+dd^c v)^n, 
    $$
    for some positive constant $\lambda>0$. Then $u\geq v$. 
\end{corollary}

We are now going to prove the following result:
\begin{theorem}\label{thm: sol big}
Assume $0\leq f\in L^p(X)$ for some $p>1$. Then there exist $\varphi\in \PSH(X,\theta)\cap L^{\infty}$ and a unique constant $c>0$ solving
\[
(\theta+dd^c \varphi)^n = cf \omega^n. 
\]
\end{theorem}
\begin{proof}
    Uniqueness of the constant $c$ follows from Corollary \ref{cor:MAconstant}.  To prove the existence we approximate $\theta$ by $\theta_j=\theta + j^{-1}\omega$, and  thanks to Theorem \ref{MA zero case} we solve
    \[
    (\theta_j +dd^c u_j)^n = c_j f\omega^n, \; \sup_X u_j=0. 
    \]
Let $\omega_{g}$ be the unique Gauduchon metric in the conformal class of $\omega$ and write $\omega_{g}^n=e^{h} \omega^n$, for some smooth function $h$. The AM-GM inequality gives the mixed Monge-Amp\`ere inequality 
    \[
    (\theta_j+dd^c u_j) \wedge \omega_{g}^{n-1} \geq c_j^{1/n} f^{1/n} \exp \left ( \frac{(n-1)h}{n} \right) \omega^n.  
    \]
   Since $\int_X  (\theta_j+dd^c u_j) \wedge \omega_g^{n-1} = \int_X  \theta_j \wedge \omega_g^{n-1}\leq \int_X ( \theta+\omega) \wedge \omega_g^{n-1}$, integrating the above over $X$, we obtain an upper bound for $c_j$. Moreover, using Lemma \ref{lem: subsol big} and Corollary  \ref{cor:MAconstant}, we see that $c_j\geq c_0>0$. Thus, extracting a subsequence and using the same trick as in Theorem \ref{thm: lambda 0}, we get a solution $\varphi \in \PSH(X,\theta)\cap L^{\infty}$.  
\end{proof}

\begin{theorem}\label{thm: higher0}
    With the same assumptions as above and assume moreover that $f=e^{\psi^+-\psi^-}\in L^p$, $p>1$, where $\psi^{\pm}$ are quasi-psh functions smooth in $X\setminus E$, where $E$ is an analytic set. Then there is a bounded solution $\varphi$ to the equation 
    \[
    (\theta+dd^c \varphi)^n =e^{b+\psi^+-\psi^-}\omega^n
    \]
    which is smooth in $X\setminus (D\cup E)$. 
\end{theorem}

Recall that $\rho \in \PSH(X,\theta)$ has analytic singularities, $\theta +dd^c \rho \geq \delta \omega$, for some $\delta>0$, and $D$ is an analytic set such that $\rho$ is smooth in $\Omega:=X\setminus D$.
\begin{proof}
    We proceed in several steps. 
    
\noindent  {\it Smooth approximation.}
 We assume without loss of generality that $\p^{\pm} \leq 0$. For a Borel function $g$ we let $g_{\e}:=\rho_{\e}(g)$ denote the Demailly regularization of $g$ (see \cite[(3.1)]{Dem94}):
 
\[
\rho_{\e}(g)(x) :=  \frac{1}{\e^{2n}} \int_{\zeta \in T_xX} g({\rm exph}_x(\zeta)) \chi\left (\frac{|\zeta|^2}{\e^2} \right ) d\lambda (\zeta). 
\]
  Since $\p^-$ is quasi-psh, the corresponding 
regularization $\p^-_{\e}$ satisfies $\p^- \leq \p^-_{\e}+A \e^2$,
 while $\p^+_{\e} \leq0$. 
 Moreover the functions $\p_\e^{\pm}$ are quasi-psh with
$dd^c \p_\e^{\pm} \geq -K^{\pm} \omega$ for uniform constants $K^{\pm} \geq 0$. In particular
$$
K^- \theta+ dd^c \left( \delta \p_\e^{-}+K^- \rho \right) \geq 
-K^-\delta \omega+K^- (\theta+dd^c \rho) = 0,
$$
so $\alpha \p_\e^{-}+ \rho \in \PSH(X, \theta)$ for all $0<\alpha\leq  \frac{\delta}{K^-}$. Up to replacing $\delta$ with $\frac{\delta}{K^-}$,  we assume in the sequel that $K^-=1$. 
 
\smallskip

We fix $0< \e \leq 1$ and set $\theta_{\e}:=\theta+\e \omega$.
 It follows from Theorem \ref{thm: lambda 0} that there exist unique constants $c_\e>0$ and  smooth
 $\theta_{\e}$-psh functions $\f_\e$   such that
 $$
 (\theta+\e \omega+dd^c \f_\e)^n=c_\e e^{\psi^+_{\e}-\psi^-_{\e}} \omega^n,
 $$
 with  $\sup_X \f_\e=0$.

\medskip

\noindent  {\it ${\mathcal C}^0$-estimates.} It follows from Jensen's inequality that $e^{\psi^+_{\varepsilon}-\psi^-_{\varepsilon}}\leq \rho_\varepsilon(f)$, hence $e^{\psi^+_{\varepsilon}-\psi^-_{\varepsilon}}$ is uniformly bounded in $L^p(\omega^n)$. From this and the same arguments in the beginning of the proof of Theorem \ref{thm: lambda 0} we can prove that $\varphi_{\varepsilon}$ is uniformly bounded and $c_{\varepsilon}$ is also uniformly bounded. 

  \medskip
 
\noindent  {\it ${\mathcal C}^2$-estimates.} 
In the sequel we establish a uniform bounds on $\Delta_{\omega} \f_\e$ on compact subsets of $\Omega = X \setminus (D \cup E)$.
% Since $\omega+dd^c \rho$ is a hermitian form in $\Omega$, the forms 
% $\beta_{\e}=\omega+dd^c \rho+\e \omega$
% and $\omega$ are uniformly comparable on compact subsets of $\Omega$, so it suffices
% to bound from above ${\rm Tr}_{\beta_{\e}}(\omega_{\e}+dd^c \f_{\e})$.
 We follow the computations of \cite[Proof of Theorem 2.1]{TW10a} and \cite{To18} 
 with a twist in order to deal with unbounded functions. 
 We use $C$ to denote various uniform constants which may be different. Consider
 $$
 H:=\log {\rm Tr}_{\omega}(\tilde{\omega})- \gamma(u)
 $$
 where 
 $$
\tilde{\omega}=\theta_{\e}+dd^c \f_{\e},
 \; \;
 u= (\f_{\e}-\rho -2 \delta \psi^-_{\varepsilon} -\inf_X \varphi_{\varepsilon}+1 )>1, 
 $$
 and $\gamma: \mathbb{R} \rightarrow \mathbb{R}$ is a smooth concave increasing function such that $\gamma(+\infty)=+\infty$.
 % which will be chosen later. 
 We are going to show that $H$ is uniformly bounded from above for an appropriate choice of
 $\gamma$. Since $u$ is uniformly bounded on compact subsets of $\Omega$, this will
 yield for each $\e>0$ and $K \subset \subset \Omega$ a uniform bound
 $$
 \Delta_{\omega}(\f_{\e})=
 {\rm Tr}_{\omega}(\theta_{\e}+dd^c \f_{\e}) - {\rm Tr}_{\omega}(\theta_{\e}) \leq C_K.
 $$
 
 We let ${g}$ denote the Riemannian metric associated to $\omega$
 and $\tilde{g}$ the one associated to $\tilde{\omega}=\theta_\e+dd^c \f_{\e}$. 
 To simplify notations we will omit the subscript $\varepsilon$ in the sequel. Since $\psi^{\pm}$ are quasi-psh, up to multiplying $\omega$ with a large constant, we can assume that $\omega+dd^c \psi^{\pm} \geq 0$. 
 
 Since $\rho \rightarrow -\infty$ on 
 %$\partial \Omega$, 
 $E$
 the maximum of $H$ is attained at some point $x_0 \in X \setminus E$.  
 We use special coordinates at this point, the same we used in the proof of Theorem \ref{thm: Lap est}:
 \[
 g_{i\bar{j}} = \delta_{ij}, \; \;  
 \frac{\partial g_{i\bar{i}}}{\partial z_j}=0 
 \; \; \text{and}\; \; \tilde{g}_{i\bar{j}}\; \text{is diagonal}.
 \]
Using the exact same type of computations and the fact that 
$$
{\rm Ric}(\tilde{\omega})= {\rm Ric}(\omega) - dd^c (\psi^+-\psi^-) \leq C \omega + dd^c \psi^-,
$$
we obtain 
\begin{equation}\label{eq: C2 estimate 3}
 \Delta_{\tilde{\omega}}
 {\rm Tr}_{\omega} (\tilde{\omega}) 
\geq   I -C {\rm Tr}_{\omega} (\omega+dd^c \psi^-) -C {\rm Tr}_{\omega} (\tilde{\omega})  {\rm Tr}_{\tilde{\omega}}  (\omega) -C  \Tr_{\tilde{\omega}}(\omega),
\end{equation}
% \end{flalign}
 with $I:=\sum_{i,j} \tilde{g}^{i \bar{i}}  \tilde{g}^{j \bar{j}} \tilde{g}_{i \bar{j}j}   \tilde{g}_{j \bar{i} \bar{j}}$.

Since $0 \geq \Delta_{\tilde{\omega}} H$  at $x_0$, it follows from \eqref{eq: C2 estimate 4}, \eqref{eq:CS trace} that
  \begin{flalign}\label{eq: C2 est final 1}
0 &\geq \Delta_{\tilde{\omega}} H
= 
 \frac{\Delta_{\tilde{\omega}} {\rm Tr}_{\omega} (\tilde{\omega})}{ {\rm Tr}_{\omega} (\tilde{\omega})}
 -\frac{| \partial {\rm Tr}_{\omega} (\tilde{\omega}) |^2_{\tilde{\omega}}}{({\rm Tr}_{\omega} (\tilde{\omega}))^2}
 -\gamma'(u) \Delta_{\tilde{\omega}}(u) - \gamma''(u) |\partial u|_{\tilde{\omega}}^2  \\
 &\geq  -\frac{C {\rm Tr}_{\omega} (\omega +dd^c \psi^-)}{{\rm Tr}_{\omega}(\tilde{\omega})} 
- \gamma'(u) (n - \delta{\rm Tr}_{\tilde{\omega}}( 3\omega+ 2dd^c \psi^- )) \nonumber \\
&-C \left (\frac{\gamma'(u)^2}{(-\gamma''(u))}+1\right)
\frac{{\rm Tr}_{\tilde{\omega}}(\omega)}{({\rm Tr}_{\omega} (\tilde{\omega}))^2}-C \frac{\Tr_{\tilde{\omega}}(\omega)}{\Tr_{\omega}(\tilde{\omega})} -C \Tr_{\tilde{\omega}}(\omega) .\nonumber
 \end{flalign}
 
We now choose (as in the proof of Theorem \ref{thm: Lap est}) the function $\gamma$ as 
   $$
 \gamma(u):= \frac{C+1}{\min(\delta,1)} u + \ln(u)
 $$
and we recall that, since $u\geq 1$, 
 $$
 \frac{C+1}{\min(\delta,1)} \leq \gamma'(u) \leq 1+ \frac{C+1}{\min(\delta,1)} 
\; \; \; \text{ and }\;  \; \;
  \frac{\gamma'(u)^2}{|\gamma''(u)|} +1 \leq C_1u^2.
 $$
By incorporating  this into \eqref{eq: C2 est final 1}  we obtain
\begin{flalign*}
	0 &\geq -\frac{C {\rm Tr}_{\omega} (\omega +dd^c \psi^-)}{{\rm Tr}_{\omega}(\tilde{\omega})} -C_2+ (C+1)(\Tr_{\tilde{\omega}}(\omega) + \Tr_{\tilde{\omega}}(\omega+dd^c \psi^-))  \\
	&-C_2 (u^2+1)
\frac{{\rm Tr}_{\tilde{\omega}}(\omega)}{({\rm Tr}_{\omega} (\tilde{\omega}))^2}-C \frac{\Tr_{\tilde{\omega}}(\omega)}{\Tr_{\omega}(\tilde{\omega})} -C \Tr_{\tilde{\omega}}(\omega).
\end{flalign*}
Using
  ${\rm Tr}_{\omega} (\omega +dd^c \psi^-) \leq {\rm Tr}_{\tilde{\omega}} (\omega +dd^c \psi^-){\rm Tr}_{\omega} (\tilde{\omega})$ we thus arrive at 
 \begin{equation}
 	\label{eq: C2 est final 2 semipositive}
 	 0 \geq {\rm Tr}_{\tilde{\omega}}(\omega) - C_2(u^2+1) \frac{{\rm Tr}_{\tilde{\omega}}(\omega)}{({\rm Tr}_{\omega} (\tilde{\omega}))^2}-C\frac{\Tr_{\tilde{\omega}}(\omega)}{\Tr_{\omega}(\tilde{\omega})}- C_2. 
 \end{equation}
 
\noindent  At the point $x_0$ we have the following alternative:
\begin{itemize}
\item  if ${\rm Tr}_{\omega}(\tilde{\omega})^2 \geq 4C_2 (u^2+1)+(4C)^2$ then 
$$
C_2(u^2+1)\frac{{\rm Tr}_{\tilde{\omega}}(\omega)}{({\rm Tr}_{\omega} (\tilde{\omega}))^2} \leq \frac{{\rm Tr}_{\tilde{\omega}}(\omega)}{4}
\; \; \text{and} \; \;
 C\frac{\Tr_{\tilde{\omega}}(\omega)}{\Tr_{\omega}(\tilde{\omega})} \leq \frac{{\rm Tr}_{\tilde{\omega}}(\omega)}{4},
$$
hence from \eqref{eq: C2 est final 2 semipositive} we get  ${\rm Tr}_{\tilde{\omega}}(\omega) \leq 2C_2$. 
Now 
\begin{equation*}%\label{eq: compare two trace}
\Tr_{\omega}(\tilde{\omega}) 
\leq n \frac{\tilde{\omega}^n}{\omega^n} \left( \Tr_{\tilde{\omega}}(\omega) \right)^{n-1}
\leq  n(2C_2)^{n-1} c e^{\psi^+-\psi^-}
\end{equation*}
yields ${\rm Tr}_{\omega}(\tilde{\omega}) \leq Ce^{\psi^+-\psi^-}$.
It follows that
\begin{flalign*}
H(x_0) &\leq \log (2C) - \psi^- -\frac{C+1}{\min(\delta,1)} (\f - \rho - 2\delta \psi^-)\\
&\leq  	\log (2C) - (C+1) (\f -\rho -\delta \psi^-) 
\leq C_3. 
\end{flalign*}
%where in the last line we have used 
%by using \eqref{eq: C0 estimate}. 
\item   If ${\rm Tr}_{\omega}(\tilde{\omega})^2 \leq 4C_2 (u^2+1)+(4C)^2$ then
$$
 H(x_0) \leq \log \sqrt{4C_2(u^2+1) +(4C)^2}- \gamma(u) \leq C_4.
 $$
 \end{itemize}
Thus $H(x_0)$ is uniformly bounded from above, yielding the desired estimate.

\medskip

\noindent  {\it Higher order estimates.}
With uniform bounds on $||\Delta_{\omega} \f_\e||_{L^{\infty}(K)}$ in hands, we
can use a complex version of Evans-Krylov-Trudinger 
%${\mathcal C}^{\alpha}$
 estimate
(see \cite[Section 4]{TW10a}) and eventually
differentiate the equation to obtain 
-using Schauder estimates-
uniform bounds,
 for each $K \subset \subset \Omega$, $0 < \beta<1$, $j \geq 0$,
$$
\sup_{\e>0} ||\f_{\e}||_{{\mathcal C}^{j,\beta}(K)}=C_{j,\beta}(K) <+\infty,
$$
 which guarantee that $\f_{\e}$ is relatively compact in ${\mathcal C}^{\infty}(\Omega)$.
 
 We now extract a subsequence $\e_j \rightarrow 0$ such that
 \begin{itemize}
 \item $c_{\e_j} \longrightarrow  c>0$;
 \item $\f_{\e_j} \rightarrow \f$ in $L^1$ with $\f \in \PSH(X,\theta)$ and $\sup_X \f=0$ (Hartogs lemma);
 \item $\f \in {\mathcal C}^{\infty}(\Omega)$ with $(\theta+dd^c \f)^n=c f\omega^n$ in $\Omega$;
 \item $|\f|\leq C$ in $X$.
 \end{itemize}
\end{proof}

\section{Singular Calabi-Yau metrics}

Let $V$ be a compact complex variety with {\it log-terminal singularities}, i.e. 
$V$ is  a normal complex space such that the canonical bundle $K_V$
is $\Q$-Cartier and for some (equivalently any) resolution of singularities
$\pi:X \rightarrow V$, we have
$$
K_X=\pi^* K_V+\sum_i a_i E_i,
$$
where the $E_i$'s are exceptional divisors with simple normal crossings,
and the rational coefficients $a_i$ 
(the discrepancies) satisfy $a_i>-1$.

Given $\phi$ a smooth metric of $K_V$ and $\sigma$ a
non vanishing  local holomorphic section
of $K_V^{\otimes r}$, we consider the ``adapted volume form''
$$
\mu_{\phi}:=\left( \frac{i^{rn^2} \sigma \wedge \overline{\sigma}}{|\sigma|^2_{r\phi}} \right)^{\frac{1}{r}}.
$$
This measure is independent of the choice of $\sigma$, and it has finite mass on $V$, since
the singularities are  log-terminal.
Given $\omega_V$ a hermitian form  on $V$ (see \cite[Section 16.3.1.2]{GZbook}) for a definition of a smooth form on a singular normal variety), there exists a unique  metric  $\phi=\phi(\omega_V)$ 
of $K_V$ such that 
$$
\omega_V^n=\mu_{\phi}.
$$
\begin{definition}
The Chern-Ricci curvature form of $\omega_V$ is ${\rm Ric}(\omega_V):=-dd^c \phi$.    
\end{definition}

%The first Chern form of the Chern connection of $\omega_V$ is a closed form 
%cohomologousto $c_1(V)$, 
%which we denote by ${\rm Ric}(\omega_V)$.
Recall that the Bott-Chern space $H_{BC}^{1,1}(V,\R)$
is the space of closed real $(1,1)$-forms modulo the image of $dd^c$ acting on real functions.
The form ${\rm Ric}(\omega_V)$ determines a class $c_1^{BC}(V)$ which 
maps to the usual Chern class $c_1(V)$ under the natural surjection $H_{BC}^{1,1}(V,\R) \rightarrow H^{1,1}(V,\R)$.

By analogy with the Calabi conjecture in K\"ahler geometry, it is natural to wonder conversely, whether
any representative $\eta \in c_1^{BC}(V)$ can be realised as the Ricci curvature
form of a hermitian metric $\omega_V$.
We provide a positive answer, as a consequence of Theorem \ref{thm: sol big} and Theorem \ref{thm: higher0}:

\begin{theorem} \label{thm:calabi}
Let $V$ be a compact hermitian variety with log terminal singularities
equipped with a hermitian form $\omega_V$.
For every smooth closed real $(1,1)$-form   $\eta$ in 
%the first Bott-Chern class 
$c_1^{BC}(V)$, there exists a unique
%unique \marginpar{unique?} 
function $\f \in \PSH(V,\omega_V)$ 
%normalized by $\sup_X \f=0$ 
such that
\begin{itemize}
\item $\f$ is globally bounded on $V$ and smooth in $V_{reg}$;
\item $\omega_V+dd^c \f$ is a hermitian form  and ${\rm Ric}(\omega_V+dd^c \f)=\eta$ in $V_{reg}$.
\end{itemize}
\end{theorem}

In particular if $c_1^{BC}(V)=0$, any hermitian form $\omega_V$ is ``$dd^c$-cohomologous'' to a 
Ricci flat hermitian current. Understanding the asymptotic behavior of these singular Ricci flat currents
near the singularities of $V$ is, as in the K\"ahler case, an important open problem.

\begin{proof}
It is classical that solving the (singular) Calabi conjecture is equivalent to
solving a complex Monge-Amp\`ere equation.
We let $\pi:X \rightarrow V$ denote a log-resolution of singularities
and observe that
$$
\pi^* \mu_{\phi}=f dV,
\; \; \text{ where } \; \; 
f=\prod_{i=1}^k |s_i|^{2a_i}
$$
has poles (corresponding to $a_i<0$) or zeros (corresponding to $a_i>0$)
 along the exceptional divisors $E_i=(s_i=0)$ and
$dV$ is a smooth volume form on $X$.

We set $\p^+=\sum_{a_i>0} 2 a_i\log|s_i|$, $\p^-=\sum_{a_i<0} 2 |a_i|\log|s_i|$,
and fix $\phi$ a smooth metric of $K_V$ such that $\eta=-dd^c \phi$.
Finding $\omega_V+dd^c \f$ such that 
${\rm Ric}(\omega_V+dd^c \f)=\eta$
is thus equivalent to solving the Monge-Amp\`ere equation
$(\omega_V+dd^c \f)^n=c \mu_{\phi}$.
Passing to the resolution this boils down to solve
$$
(\theta+dd^c \tilde{\f})^n
=c e^{\p^+-\p^-} dV
$$
on $X$, where $\theta=\pi^* \theta_V$ and $\tilde{\f}=\f \circ \pi \in \PSH(X,\theta)$.

Since $\theta$ is semi-positive and big, and since $\p^{\pm}$ are quasi-plurisubharmonic functions
which are smooth in $X^0=\pi^{-1}(V_{reg})$, it follows from
Theorem \ref{thm: sol big} and Theorem \ref{thm: higher0}
that there exists a  solution
$\tilde{\f}$ with all the required properties.
Then $\f=\pi_* \tilde{\f}$ is the function we were looking for.
\end{proof}

\bibliographystyle{alpha}
\bibliography{Biblio}

\newcommand{\etalchar}[1]{$^{#1}$}
\begin{thebibliography}{GPTW24}

\bibitem[{\AA}CK{\etalchar{+}}09]{ACKPZ09}
P.~{\AA}hag, U.~Cegrell, S.~Ko{\l}odziej, H.~H. Ph\d{a}m, and A.~Zeriahi.
\newblock Partial pluricomplex energy and integrability exponents of
  plurisubharmonic functions.
\newblock {\em Adv. Math.}, 222(6):2036--2058, 2009.

\bibitem[Aub78]{Aub78}
T.~Aubin.
\newblock \'{E}quations du type {M}onge-{A}mp\`ere sur les vari\'{e}t\'{e}s
  k\"{a}hl\'{e}riennes compactes.
\newblock {\em Bull. Sci. Math. (2)}, 102(1):63--95, 1978.

\bibitem[BEGZ10]{BEGZ10}
S.~Boucksom, P.~Eyssidieux, V.~Guedj, and A.~Zeriahi.
\newblock Monge-{A}mp\`ere equations in big cohomology classes.
\newblock {\em Acta Math.}, 205(2):199--262, 2010.

\bibitem[Ber19]{Ber19}
R.~J. Berman.
\newblock From {M}onge-{A}mp\`ere equations to envelopes and geodesic rays in
  the zero temperature limit.
\newblock {\em Math. Z.}, 291(1-2):365--394, 2019.

\bibitem[BGL24]{BGL24}
S.~Boucksom, Vincent Guedj, and C.~H. Lu.
\newblock {Volumes of Bott-Chern classes}.
\newblock {\em Preprint arXiv:2406.01090}, 2024.

\bibitem[B{\l}o05]{Bl05China}
Z.~B{\l}ocki.
\newblock On uniform estimate in {C}alabi-{Y}au theorem.
\newblock {\em Sci. China Ser. A}, 48(suppl.):244--247, 2005.

\bibitem[B{\l}o11]{Bl11China}
Z.~B{\l}ocki.
\newblock On the uniform estimate in the {C}alabi-{Y}au theorem, {II}.
\newblock {\em Sci. China Math.}, 54(7):1375--1377, 2011.

\bibitem[CC21]{ChCh21a}
X.X. Chen and J.~Cheng.
\newblock On the constant scalar curvature {K}\"{a}hler metrics ({I})---{A}
  priori estimates.
\newblock {\em J. Amer. Math. Soc.}, 34(4):909--936, 2021.

\bibitem[Che87]{Cher87}
P.~Cherrier.
\newblock {\'Equations de Monge-Amp\`ere sur les vari\'et\'es Hermitiennes
  compactes}.
\newblock {\em Bull. Sci. Math.}, 2(343--385.), 1987.

\bibitem[Chi24]{Chi24}
I.~Chiose.
\newblock On the invariance of the total {Monge{\textendash}Amp\`ere} volume of
  {Hermitian} metrics.
\newblock {\em Annales de la Facult\'e des sciences de Toulouse :
  Math\'ematiques}, Ser. 6, 33(3):575--579, 2024.

\bibitem[CZ19]{CZ19}
J.~Chu and B.~Zhou.
\newblock Optimal regularity of plurisubharmonic envelopes on compact
  {H}ermitian manifolds.
\newblock {\em Sci. China Math.}, 62(2):371--380, 2019.

\bibitem[DDNL21]{DDL4}
T.~Darvas, E.~Di~Nezza, and C.~H. Lu.
\newblock Log-concavity of volume and complex {M}onge-{A}mp\`ere equations with
  prescribed singularity.
\newblock {\em Math. Ann.}, 379(1-2):95--132, 2021.

\bibitem[Dem94]{Dem94}
J.-P. Demailly.
\newblock Regularization of closed positive currents of type {$(1,1)$} by the
  flow of a {C}hern connection.
\newblock In {\em Contributions to complex analysis and analytic geometry},
  Aspects Math., E26, pages 105--126. Friedr. Vieweg, Braunschweig, 1994.

\bibitem[DK12]{DK12}
S.~Dinew and S.~Ko{\l}odziej.
\newblock Pluripotential estimates on compact {H}ermitian manifolds.
\newblock In {\em Advances in geometric analysis}, volume~21 of {\em Adv. Lect.
  Math. (ALM)}, pages 69--86. Int. Press, Somerville, MA, 2012.

\bibitem[EGZ09]{EGZ09}
P.~Eyssidieux, V.~Guedj, and A.~Zeriahi.
\newblock Singular {K}\"{a}hler-{E}instein metrics.
\newblock {\em J. Amer. Math. Soc.}, 22(3):607--639, 2009.

\bibitem[GL10]{GL10}
B.~Guan and Q.~Li.
\newblock Complex {M}onge-{A}mp\`ere equations and totally real submanifolds.
\newblock {\em Adv. Math.}, 225(3):1185--1223, 2010.

\bibitem[GL22]{GL22}
V.~Guedj and C.H. Lu.
\newblock Quasi-plurisubharmonic envelopes 2: {B}ounds on {M}onge-{A}mp\`ere
  volumes.
\newblock {\em Algebr. Geom.}, 9(6):688--713, 2022.

\bibitem[GL23]{GL23Crelle}
V.~Guedj and C.H. Lu.
\newblock Quasi-plurisubharmonic envelopes 3: {S}olving {M}onge-{A}mp\`ere
  equations on hermitian manifolds.
\newblock {\em J. Reine Angew. Math.}, 800:259--298, 2023.

\bibitem[GL24]{GL24}
V.~Guedj and C.~H. Lu.
\newblock {Quasi-plurisubharmonic envelopes 1: Uniform estimates on K\"ahler
  manifolds}.
\newblock {\em arxiv:2106.04273, accepted on J. Eur. Math. Soc.}, 2024.

\bibitem[GLZ19]{GLZ19}
V.~Guedj, C.~H. Lu, and A.~Zeriahi.
\newblock Plurisubharmonic envelopes and supersolutions.
\newblock {\em J. Differential Geom.}, 113(2):273--313, 2019.

\bibitem[GP24]{GP24}
B.~Guo and D.H. Phong.
\newblock {On $L^\infty$ estimates for fully non-linear partial differential
  equations}.
\newblock {\em Annals of Mathematics}, 200(1):365 -- 398, 2024.

\bibitem[GPT23]{GPT23}
B.~Guo, D.H. Phong, and F.~Tong.
\newblock On {{\(L^\infty\)}} estimates for complex {Monge}-{Amp{\`e}re}
  equations.
\newblock {\em Ann. Math. (2)}, 198(1):393--418, 2023.

\bibitem[GPTW24]{GPTW}
B.~Guo, D.~H. Phong, F.~Tong, and C.~Wang.
\newblock On {{\(L^{\infty}\)}} estimates for {Monge}-{Amp{\`e}re} and
  {Hessian} equations on nef classes.
\newblock {\em Anal. PDE}, 17(2):749--756, 2024.

\bibitem[GT01]{GT01}
D.~Gilbarg and N.S. Trudinger.
\newblock {\em Elliptic partial differential equations of second order}.
\newblock Classics in Mathematics. Springer-Verlag, Berlin, 2001.
\newblock Reprint of the 1998 edition.

\bibitem[GZ17]{GZbook}
V.~Guedj and A.~Zeriahi.
\newblock {\em Degenerate complex {M}onge-{A}mp\`ere equations}, volume~26 of
  {\em EMS Tracts in Mathematics}.
\newblock European Mathematical Society (EMS), Z\"{u}rich, 2017.

\bibitem[HL11]{HanLin11}
Q.~Han and F.~Lin.
\newblock {\em Elliptic partial differential equations}, volume~1 of {\em
  Courant Lect. Notes Math.}
\newblock New York, NY: Courant Institute of Mathematical Sciences; Providence,
  RI: American Mathematical Society (AMS), 2nd ed. edition, 2011.

\bibitem[KN15]{KN15Phong}
S.~Ko{\l}odziej and N.-C. Nguyen.
\newblock Weak solutions to the complex {M}onge-{A}mp\`ere equation on
  {H}ermitian manifolds.
\newblock In {\em Analysis, complex geometry, and mathematical physics: in
  honor of {D}uong {H}. {P}hong}, volume 644 of {\em Contemp. Math.}, pages
  141--158. Amer. Math. Soc., Providence, RI, 2015.

\bibitem[KN19]{KN19}
S.~Ko{\l}odziej and N.-C. Nguyen.
\newblock Stability and regularity of solutions of the {M}onge-{A}mp\`ere
  equation on {H}ermitian manifolds.
\newblock {\em Adv. Math.}, 346:264--304, 2019.

\bibitem[Ko{\l}98]{Kol98}
S.~Ko{\l}odziej.
\newblock The complex {M}onge-{A}mp\`ere equation.
\newblock {\em Acta Math.}, 180(1):69--117, 1998.

\bibitem[LU68]{LU}
O.~Ladyzenskaya and N.~Uralsteva.
\newblock Linear and quasilinear elliptic partial differential equations.
\newblock {\em Academic Press}, 1968.

\bibitem[Ngu16]{Ng16AIM}
N.-C. Nguyen.
\newblock The complex {M}onge-{A}mp\`ere type equation on compact {H}ermitian
  manifolds and applications.
\newblock {\em Adv. Math.}, 286:240--285, 2016.

\bibitem[Sz{\'{e}}18]{Sze18}
G.~Sz{\'{e}}kelyhidi.
\newblock Fully non-linear elliptic equations on compact {H}ermitian manifolds.
\newblock {\em J. Differential Geom.}, 109(2):337--378, 2018.

\bibitem[T{\^{o}}18]{To18}
T.~D. T{\^{o}}.
\newblock Regularizing properties of complex {M}onge-{A}mp\`ere flows {II}:
  {H}ermitian manifolds.
\newblock {\em Math. Ann.}, 372(1-2):699--741, 2018.

\bibitem[TW10a]{TW10}
V.~Tosatti and B.~Weinkove.
\newblock The complex {M}onge-{A}mp\`ere equation on compact {H}ermitian
  manifolds.
\newblock {\em J. Amer. Math. Soc.}, 23(4):1187--1195, 2010.

\bibitem[TW10b]{TW10a}
V.~Tosatti and B.~Weinkove.
\newblock Estimates for the complex {Monge}-{Amp{\`e}re} equation on
  {Hermitian} and balanced manifolds.
\newblock {\em Asian J. Math.}, 14(1):19--40, 2010.

\bibitem[Yau78]{Yau78}
S.-T. Yau.
\newblock On the {R}icci curvature of a compact {K}\"{a}hler manifold and the
  complex {M}onge-{A}mp\`ere equation. {I}.
\newblock {\em Comm. Pure Appl. Math.}, 31(3):339--411, 1978.

\end{thebibliography}

\end{document}